\documentclass[a4paper,11pt,leqno]{amsart}
\usepackage{amsmath}
\usepackage{amsthm}
\usepackage{amsfonts}
\usepackage{amssymb}
\usepackage[mathscr]{eucal}
\usepackage{eucal}

\usepackage{graphicx}
\usepackage{mathrsfs}
\numberwithin{equation}{section}
\theoremstyle{plain}
\newtheorem{thm}{Theorem}[section]
\newtheorem{lem}[thm]{Lemma}
\newtheorem{prop}[thm]{Proposition}
\newtheorem{cor}[thm]{Corollary}
\newtheorem{rem}[thm]{Remark}

\theoremstyle{remark}

\theoremstyle{definition}

\def\B{\mathbb{B}}

\def\R{\mathbb{R}}
\def\N{\mathbb{N}}
\def\S{\Sigma}

\def\e{\epsilon}
\def\g{\gamma}
\def\d{\delta}
\def\D{\Delta}

\def\G{\Gamma}
\def\a{\alpha}

\newcommand{\Vaisala}{V\"ais\"al\"a}
\DeclareMathOperator{\diam}{diam}
\DeclareMathOperator{\dist}{dist}

\DeclareMathAlphabet{\mathscr}{OT1}{pzc}{m}{it}

\begin{document}

\title{Quasisymmetric Spheres over  Jordan Domains}
\date{\today}
\author{Vyron Vellis}

\author{Jang-Mei Wu}

\address{Department of Mathematics, University of Illinois,  1409 West Green Street, Urbana, IL 61820, USA}

\email{vellis1@illinois.edu}

\email{wu@math.uiuc.edu}

\address{Department of Mathematics, University of Illinois,  1409 West Green Street, Urbana, IL 61801, USA}

\thanks{Research supported in part by the NSF grant DMS-1001669.}


\begin{abstract}
Let $\Omega$ be a planar Jordan domain. We consider double-dome-like surfaces $\S$ defined by graphs of  functions of $\dist(\cdot,\partial \Omega)$ over $\Omega$. The goal is to find the right conditions on the geometry of 
the base $\Omega$ and the growth of the height so that $\S$ is a quasisphere, or quasisymmetric to $\mathbb{S}^2$. An internal uniform chord-arc condition on the constant distance sets to $\partial \Omega$, coupled with a mild 
growth condition on the height, gives a close-to-sharp answer. Our method also produces new examples of quasispheres in $\R^n$, for any $n\ge 3$.

\end{abstract}

\maketitle
\date{\today}

\section{Introduction}\label{sec:intro}
Images of $\mathbb S^{n}$  under quasiconformal homeomorphisms of $\mathbb R^{n+1}$ are called quasispheres.
In $\R^2$, quasicircles can be described completely in geometrical terms (\cite{Ah}, \cite{Gehring-characterization}, \cite{TuVa}).
In higher dimensions the only known characterization is due to Gehring \cite{Gehext3} and V\"ais\"al\"a \cite{Vais2}:
a topological $n$-sphere $\Sigma$ in $\R^{n+1}$ is a quasisphere if and only if the bounded component and the unbounded component of $\R^{n+1} \setminus \S$ are quasiconformally equivalent to $\mathbb{B}^{n+1}$ and 
$\R^{n+1} \setminus \overline{\mathbb B^{n+1}}$, respectively. Intriguing examples of quasispheres have been constructed drawing ideas from harmonic analysis, conformal dynamics and classical geometric topology (\cite{bishop},
\cite{DToro}, \cite{Lewis}, \cite{Meyer}, \cite{PW}). The basic question of a geometric characterization of quasispheres remains.

More generally, metric spaces that are quasisymmetrically homeomorphic to $\mathbb{S}^{n}$ are called quasisymmetric spheres. An intricate characterization of quasisymmetric $2$-spheres has been found by Bonk and Kleiner 
\cite{BonK}; as a consequence, the Ahlfors $2$-regularity together with the local linear connectivity on a metric $2$-sphere suffices. However, little is known about quasisymmetric $n$-spheres when $n\geq 3$.

In this article, we examine double-dome-like surfaces in $\mathbb R^3$ defined by graphs of functions on a Jordan domain $\Omega \subset \R^2$,
\[\Sigma(\Omega, \varphi) =\{(x,z) \colon x\in \overline{\Omega},\, z=\pm \,\varphi(\dist(x,\partial \Omega)) \},\]
where $\varphi(\cdot)$ is a continuous increasing function on $[0,\infty)$ with $\varphi(0)=0$.
Our aim is to find the right conditions on the geometry of the base $\Omega$ and the growth of the gauge $\varphi$ in order for these surfaces to be quasispheres, or quasisymmetric spheres.

Suppose that $\Omega$ is a planar Jordan domain and $\mathbb B^3$ is the open unit ball.
Gehring \cite{Geh2} showed that the slit domain  $\R^3\setminus \overline{\Omega}$ is quasiconformally homeomorphic to $\R^3\setminus \overline{\mathbb B^3}$
if and only if $\Omega$ is a quasidisk;  {\Vaisala} proved \cite{Vais4} that
the infinite cylinder $\Omega \times \R$ is quasiconformally equivalent to $\mathbb B^3$ if and only if $\Omega$ satisfies an internal chord arc condition.
A slit domain may be regarded as the complement of $\Sigma(\Omega, \varphi)$ when $\varphi \equiv 0$ in the previous setting, and a cylindrical domain may be regarded as the domain enclosed by $\Sigma(\Omega, \varphi)$ by 
choosing  $\varphi \equiv \infty$. In this spirit and for  $\varphi(t)=t$, we have the following.

\begin{thm}\label{thm:quasisphere_t}
Suppose that $\Omega$ is a Jordan domain in $\R^2$. Then the surface
\[\S(\Omega,t)= \{(x,z) \colon    x\in \overline{\Omega},\, z=\pm \,\text{dist}(x, \partial \Omega)\}\]
is a quasisphere if and only if $\partial \Omega$ is a quasicircle.
\end{thm}

The method of constructing quasispheres in Theorem \ref{thm:quasisphere_t} extends inductively to all dimensions; see Theorem \ref{thm:quasispheres_n}.

For $\alpha>1$, surfaces $\S(\Omega,t^\alpha)= \{(x,z) \colon    x\in \overline{\Omega},\, z=\pm \,(\text{dist}(x, \partial \Omega))^\alpha\}$ are not linearly locally connected, therefore are not quasisymmetric spheres.

V\"ais\"al\"a \cite{Vais3} has shown that the product $\gamma \times I$ of a Jordan arc $\gamma$ and  an interval $I$ is quasisymmetric embeddable in $\R^2$ if and only if $\G$ satisfies the chord-arc condition.

When $0<\alpha<1$, the part of the surface $\S(\Omega,t^{\alpha})$ near $\R^2 \times \{0\}$ resembles the product of an arc with an interval locally. On the other hand, the double-dome-like surface envelops  the interior 
$\Omega$ above and below. Therefore some form of internal uniform chord-arc condition is expected in order for $\S(\Omega,t^{\alpha})$ to be quasisymmetric to $\mathbb S^2$.

A Jordan domain $\Omega $ is said to have the \emph{level chord-arc property} if the $\epsilon$-distance sets to $\partial \Omega$,
\[\gamma_{\epsilon}=\{x\in \Omega\colon \text{dist} (x,\partial \Omega)=\epsilon \},\]
are  uniform $c$-chord-arc curves for some $c> 1$ and \emph{all} sufficiently small $\epsilon>0$. We define a class of gauges
\begin{align*}
\mathcal{F}= \{&\varphi\colon[0,\infty)\to [0,\infty)\colon \text{ a homeomorphism with} \,\, \varphi(0)=0,\\
&\liminf_{t\to 0}\varphi(t)/t > 0 \text{ and }\, \varphi \, \text{is Lipschitz on }\, [r ,+\infty) \text{ for all }r>0\}.
\end{align*}
This class includes  $t^\alpha$ with $0<\alpha <1$. We now state the main theorem.

\begin{thm}\label{thm:LCA-QS-main}
Let $\Omega$ be a planar Jordan domain. Then the surface
$\S(\Omega,\varphi)$ is quasisymmetric to $\mathbb{S}^2$ for every $\varphi$ in $\mathcal{F}$ if and only if
$\Omega$ has the level chord-arc property.
\end{thm}

In general, a Lipschitz domain may contain a sequence $\g_{\e_n}$ of constant distance sets with $\e_n \to 0$,  each of which is a Jordan curve containing a cusp, hence not a quasicircle; see Remark 5.2 in \cite{VW}.

What are the intrinsic characteristics of a Jordan domain that has the level chord-arc property? A flatness module $\zeta_{\gamma}$ measuring the deviation of subarcs of a Jordan curve $\gamma$ from their chords in a uniform 
and scale-invariant way has been defined in \cite{VW} for this purpose.

From here onward, given $x,y$ on a Jordan curve $\gamma$, we take $\gamma(x,y)$ to be the subarc of $\gamma$ connecting $x$ and $ y$ of smaller diameter, or  either subarc when both have the same diameter; given $x,y$ in a 
Euclidean space, we denote by $l_{x,y}$ the infinite line containing  $x$ and $y$. Set \[ \zeta_{\gamma}(x,y) = \frac{1}{|x-y|}\sup_{z\in\gamma(x,y)} \dist(z,l_{x,y})\] and define the \emph{flatness of $\gamma$} to be
\[ \zeta_{\gamma} = \lim_{r \to 0}\sup_{x,y\in\gamma , |x-y|\leq r} \zeta_{\gamma}(x,y). \]

The connection between the flatness of $\gamma$ and the level chord-arc property has been established  in \cite{VW}.

\begin{thm}[{\cite[Theorem 1.2, Theorem 1.3]{VW}}] \label{thm:VW-1/2}
Let $\Omega$ be a planar Jordan domain and suppose that $\partial \Omega$ is a chord-arc curve with flatness $\zeta_{\partial \Omega}  <1/2$. Then $\Omega$ has the level chord-arc property. On the other hand, there exists a 
Jordan domain $\Omega$ whose boundary $\partial \Omega$ is a chord-arc with flatness  $\zeta_{\partial\Omega} = 1/2$  which does not satisfy the level chord-arc property.
\end{thm}

Combining Theorem \ref{thm:LCA-QS-main} and Theorem \ref{thm:VW-1/2} with Lemma \ref{lem:lqc1} below, we may obtain the following.

\begin{cor}\label{cor:flatness-QSspheres-1/2}
Suppose that $\Omega$ is a planar Jordan domain whose boundary is a chord-arc curve  with flatness $\zeta_{\partial \Omega} <1/2$. Then the surface $\Sigma(\Omega, t^\alpha)$ is quasisymmetric to $\mathbb S^2$ for all 
$\alpha \in (0,1]$. On the other hand, there exist Jordan domains $\Omega$ in $\R^2$ whose boundaries are chord-arc curves with flatness $\zeta_{\partial \Omega} =1/2$, to which the associated surfaces $\Sigma(\Omega, t^\alpha)$ 
are not quasisymmetric to $\mathbb S^2$ for any $\alpha \in (0,1)$.
\end{cor}

\medskip

The growth condition near $0$ imposed on the gauges in $\mathcal F$ is essential. In fact, if $\liminf_{t\to 0} \varphi(t)/t =0$ then the double-dome-like surface $\S(\Omega,\varphi)$ is not linearly locally connected therefore 
not a quasisymmetric sphere, for any Jordan domain $\Omega$. The Lipschitz condition away from $0$, on the other hand, is added to tidy up the statements. For example, the surface $\S(B^2(0,1),\varphi)$ associated to the gauge
\begin{equation*}
\varphi(t) =
\begin{cases}
1 - \sqrt{1-t} &\quad t\in[0,1],\\
t              &\quad t\in[1,+\infty),
\end{cases}
\end{equation*}
is not  quasisymmetric to $\mathbb S^2$, however  $\S(B^2(0,2),\varphi)$ is.

\medskip

In Section \ref{sec:levelsets}, we discuss properties of constant distance sets to Jordan curves. Starting from a $2$-dimensional double-dome-like surface constructed over a planar quasidisk, we build quasispheres in all 
dimensions in Section \ref{sec:quasispheres}. The proof of Theorem \ref{thm:quasisphere_t} is given in Section \ref{sec:proof_quasisphere_t}. In Section \ref{sec:LLC+LQC}, we establish a relation between the linear local 
connectedness of the surface $\S(\Omega,\varphi)$ and the level quasicircle property of the domain $\Omega$. The proofs of Theorem \ref{thm:LCA-QS-main} and Corollary \ref{cor:flatness-QSspheres-1/2} are completed in Section 
\ref{sec:heightdistance}.

\section{Preliminaries}\label{sec:prelim}
A homeomorphism $f\colon D\to D'$ between two domains in $ \mathbb{R}^n$  is called $K$-\emph{quasiconformal}  if it is orientation preserving, belongs to $ W_{\text{loc}}^{1,n}(D)$, and satisfies the distortion inequality
\[ |Df(x)|^n \le K J_f(x) \quad \text{a. e.} \,\,\, x \in D, \]
where $Df$ is the formal differential matrix and $J_f$ is the Jacobian.

A Jordan curve $\g$ in $\mathbb{R}^2$ is called a $K$-\emph{quasicircle} if it is the image of the unit circle $\mathbb S^1$ under a $K$-quasiconformal homeomorphism of $\mathbb{R}^2$. A geometric characterization due to Ahlfors 
\cite{Ah} states that a Jordan curve $\g$ is a $K$-quasicircle if and only if it satisfies the \emph{2-point condition}:
\begin{equation}\label{eq:3pts}
\text{there exists } C>1 \text{ such that for all }  x,y \in \g, \, \,\diam{\g(x,y)} \leq C|x-y|,
\end{equation}
where the distortion $K$ and the $2$-point constant $C$ are quantitatively related.

Images of $\mathbb S^{n}$  under quasiconformal homeomorphisms of $\mathbb R^{n+1}, n\ge 2,$ are called \emph{quasispheres}. The only known characterization of quasispheres is  due to Gehring  for $n=3$ in 1965, and to 
V\"ais\"al\"a   for all $n\ge 3$ in 1984.

\begin{thm}[{\cite[Theorem 5.9]{Vais2}, \cite[Theorem]{Gehext3}}]\label{thm:Gehring-Vaisala-characterization}
Let $n \geq 3$ and  $\mathcal S$ be a topological $(n-1)$-sphere in ${\mathbb{R}^n}$. Suppose that  components of ${\mathbb{R}^n} \setminus \mathcal S$ are $K$-quasiconformal to $\mathbb{B}^n$ and 
$\R^n \setminus \overline{\mathbb B^n}$ respectively. Then there is a $K'$-quasiconformal homeomorphism of ${\mathbb{R}^n}$ that maps $\mathcal S$ onto $\mathbb{S}^{n-1}$, where constant $K'>1$ depends only on $K$ and $n$.
\end{thm}

Notice that this theorem is false when $n=2$.

An embedding $f$ of a metric space $(X,d_X)$ into a metric space $(Y,d_Y)$ is said to be $L$-\emph{bi-Lipschitz} if there exists $L\geq 1$ such that for any $x, y \in X$
\[\frac{1}{L}d_X(x,y) \leq d_Y(f(x),f(y)) \leq Ld_X(x,y). \]

A rectifiable Jordan curve $\g$ in $\mathbb{R}^2$ is called a $c$-\emph{chord-arc curve} if there exists $c>1$ such that  for any $x,y \in \g$, the length of the shorter arc  ${\g}'(x,y)$ in $\g\setminus \{x,y\}$ satisfies
\[ \ell ({\g}'(x,y)) \leq c |x-y|. \]
It is straightforward to see that a rectifiable curve $\g$ is a $c$-chord-arc curve if and only if it satisfies
\[ \ell (\g(x,y)) \leq C |x-y| \]
for all $x,y \in \g$ and some $C >1$;  here constants $c$ and $ C$ are quantitatively related.

Every $c$-chord-arc curve is, in fact, the image of $\mathbb S^1$ under an $L$-bi-Lipschitz homeomorphism of $\R^2$, where the constants $c$ and $L$ are quantitatively related; see\cite[p. 23]{Tukia-ext} and 
\cite[Proposition 1.13]{JeK}.

An embedding $f$ of a metric space $(X,d_X)$ into a metric space $(Y,d_Y)$ is said to be $\eta$-\emph{quasisymmetric} if there exists a homeomorphism $\eta \colon [0,\infty) \rightarrow [0,\infty)$ such that for all $x,a,b \in X$ 
and $t>0$ with $d_X(x,a) \leq t d_X(x,b)$,
\[d_Y(f(x),f(a)) \leq \eta(t)d_Y(f(x),f(b)). \]

A metric $n$-sphere $\mathcal S$ that is quasisymmetrically homeomorphic to  $\mathbb S^n$ is called a \emph{quasisymmetric sphere} when $n\ge 2$, and a \emph{quasisymmetric circle} when $n=1$.

Beurling and Ahlfors \cite{BerAhl} showed that a planar Jordan curve is a quasisymmetric circle if and only if it is a quasicircle. Tukia and Vaisala \cite{TuVa} proved that a metric $1$-sphere is a quasisymmetric circle if and 
if it is doubling and bounded turning.

The notion of linear local connectivity generalizes the $2$-point condition on curves to general sets. A set $X \subset \R^n$ is $\lambda$-\emph{linearly locally connected} (or $\lambda-\text{LLC}$) for $\lambda\geq 1$ if the 
following two conditions are satisfied.
\begin{enumerate}
\item ($\lambda-\text{LLC}_1$) If $x\in X$, $r>0$ and $y_1,y_2 \in B^n(x,r)\cap X$, then there exists a continuum $E\subset B^n(x,\lambda r)\cap X$ containing $y_1,y_2$.
\item ($\lambda-\text{LLC}_2$) If $x\in X$, $r>0$ and $y_1,y_2 \in X \setminus B^n(x,r)$, then there exists a continuum $E\subset X \setminus B^n(x,r/\lambda)$ containing $y_1,y_2$.
\end{enumerate}

Linear local connectivity was first studied in the work of Gehring and V\"ais\"al\"a \cite{GehVai} and appeared, under the term \emph{strong local connectivity}, in a paper of Gehring \cite{Gehext3}. In the latter, a set 
$X \subset \R^n$ is said to be strongly locally connected if it satisfies (1) and (2)  for \emph{all} $x\in\R^n$ instead of only for those $x\in X$.  Walker \cite{Walker} showed that any quasicircle is strongly locally connected.

Gehring and V\"ais\"al\"a \cite{GehVai} proved that if a domain $D \subset \R^n$ is quasiconformally equivalent to $\mathbb{B}^n$ then its complement  $\mathbb{R}^n \setminus \overline{D}$ must be $\text{LLC}$. It is easy to 
check that the $\text{LLC}$ property is preserved under quasisymmetry. As a consequence, every surface quasisymmetric to $\mathbb{S}^n$ or $\R^n$ satisfies the $\text{LLC}$ property.

A metric space $X$ is said to be \emph{Ahlfors $Q$-regular} if there is a constant $C>1$ such that the $Q$-dimensional Hausdorff measure $\mathcal{H}^Q$ of every open ball $B(a,r)$ in $X$ satisfies
\begin{equation}\label{eq:2regdefn}
C^{-1}r^Q \leq \mathcal{H}^Q(B(a,r)) \leq Cr^Q,
\end{equation}
when $0<r\leq \diam{X}$.

Bonk and Kleiner found in \cite{BonK} an intrinsic characterization of quasisymmetric $2$-spheres and then derived a readily applicable sufficient condition.

\begin{thm}[{\cite[Theorem 1.1, Lemma 2.5]{BonK}}]\label{thm:BonkKleiner}
Let $X$ be an Ahlfors $2$-regular metric space homeomorphic to $\mathbb{S}^2$. Then $X$ is quasisymmetric to  $\mathbb{S}^2$ if and only if $X$ is LLC.
\end{thm}

For $x \in \R^n$ and $r>0$, define $B^n(x,r) = \{ y \in \mathbb{R}^n \colon |x-y| < r\}$ and $S^{n-1}(x_0,r) = \partial B^n(x_0,r)$; in particular,  $\mathbb{B}^n = B^n(0,1)$ and $\mathbb{S}^{n-1} = \partial \mathbb{B}^n$. 
In addition, let
\[ \R^n_+ = \{x=(x_1,\ldots, x_n)  \in \R^n \colon x_n \leq0 \}\]
be the upper half-space of $\R^n$ and $\R^n_- = \{x  \in \R^n \colon x_n \leq 0 \}$ be the lower half-space. For any $a = (a_1,\dots,a_n)\in \R^n$, denote by
\[ \pi(a) = (a_1,\dots,a_{n-1},0) \]
the projection of $a$ on the hyperplane $\R^{n-1}\times\{0\}$. For $x,y \in \R^n$, denote by $[x,y]$ the line segment having end points $x,y$ and by $l_{x,y}$ the infinite line containing the points $x,y$.

In the following, we write $u\lesssim v$ (resp. $u \simeq v$) when the ratio $u/v$ is bounded above  (resp. bounded above and below) by positive constants. These constants may vary, but are described in each occurrence.

\bigskip

\section{Level Sets to a Jordan Curve}\label{sec:levelsets}
We discuss properties of constant distance sets to a Jordan curve that will be used in the sequel. The main reference is \cite{VW}.

Let $\Omega$ be a planar Jordan domain with boundary $\G$. We define for each $\e>0$,  the $\e$-level set
\[ \g_{\e} = \{ x \in \Omega \colon \dist(x,\G) = \e\}.\]
In general $\g_\e$ need not be connected and if connected need not be a curve,  see \cite[Figure 1]{VW}. Properties of $\Omega$  ensuring  the $\e$-level sets to be Jordan curves, or uniform quasicircles, or uniform chord-arc 
curves for \emph{all} sufficiently small $\e$ have been studied in \cite{VW}.

We say $\Omega$ has the \emph{level Jordan curve  property} (or LJC \emph{property}), if there exists $\e_0>0$ such that the level set $\g_{\e}$ is a Jordan curve for every $0\leq \e \leq \e_0$. The Jordan domain $\Omega$ is 
said to have the \emph{level quasicircle property} (or LQC \emph{property}), if there exist $\e_0>0$ and $K\geq 1$ such that the level set $\g_{\e}$ is a $K$-quasicircle  for every $0\leq \e \leq \e_0$. Finally, the Jordan 
domain $\Omega$ is said to have the \emph{level chord-arc property} (or LCA \emph{property}), if there exist $\e_0>0$ and $C\geq 1$ such that $\g_{\e}$ is a $C$-chord-arc curve  for every $0\leq \e \leq \e_0$.

Sufficient conditions for a domain $\Omega$ to satisfy the LJC property or the LQC property can be given in terms of the chordal flateness of $\partial \Omega$.

\begin{prop}[{\cite[Theorem 1.1, Theorem 1.2]{VW}}]
Let $\Omega$ be a Jordan domain.
\begin{enumerate}
\item If there exists $r_0>0$ such that $\zeta_{\partial\Omega}(x,y)<1/2$ for all $x,y\in\partial\Omega$ with $|x-y|<r_0$, then $\Omega$ has the LJC property.
\item If $\zeta_{\partial\Omega} < 1/2$, then $\Omega$ has the LQC property.
\end{enumerate}
\end{prop}

In article \cite{VW}, a Jordan curve $\G$ is said to have  one of the properties LJC, LQC, or LCA,  if both components of $\R^2\setminus\G$ have that property.

A connection between the LQC property and the LCA property on $\G$ is given in \cite[Theorem 1.3]{VW}: \emph{A Jordan curve $\G$ has the \emph{LCA} property if and only if it is a chord-arc and has the \emph{LQC} property}. 
Since the proof of the LCA property for either component of  $\R^2 \setminus \G$ requires only the LQC property of the same component, the following result can be concluded.

\begin{prop}\label{prop:lqc+chord-arc}
A Jordan domain $\Omega$ in the plane satisfies the level chord-arc property if and only if it has the level quasicircle property and $\partial\Omega$ is a chord-arc curve.
\end{prop}

\medskip

We state two basic properties related to  the distance function  which are used repeatedly throughout the paper. For these and more, see \cite[p. 216]{VW}.

Assume that $\Omega$ is a planar Jordan domain and that points $x, y \in \Omega$ and $x', y' \in \partial \Omega$ satisfy $|x-x'|=\dist(x,\partial \Omega)$ and $|y-y'|= \dist(y,\partial \Omega)$. Then

(\emph{non-crossing})\,  the line segments $[x,x']$ $[y,y']$ do not intersect except perhaps at their endpoints, unless $x,y,x',y'$ are collinear;

(\emph{monotonicity})\,  the distance function $\dist(\cdot\, ,\partial \Omega)$ is strictly monotone on $[x,x']$. Moreover, if $z\in [x,x']$ then $\dist(z,\partial \Omega) = |z-x'|$.

\medskip

\begin{lem}[{\cite[Lemma 4.1]{VW}}]\label{lem:orientation}
Let $\Omega$ be a planar Jordan domain and $\Lambda \subset \partial \Omega$ be a closed subarc. Assume that for some $\e> 0$ the level set $\g_{\e}$ is a Jordan curve. Then the set
\[ \{ x \in \g_{\e} \colon \dist(x,\Lambda) = \e\}, \]
if nonempty, is a subarc of $\g_{\e}$.
\end{lem}

In general $\dist(z, \g_\e)$ from points $z \in \G=\partial \Omega $ to the level sets $\g_\e$ need not have magnitude  $\e$, even when $\g_{\e}$ are Jordan curves. The  curve 
$\G=\{(s,t)\colon |t|=s^2, 0\leq s <1\} \cup  \{(1,t) \colon -1\leq t\leq 1\}$ in the Cartesian coordinates $(s,t)$ of the plane, is an example. This would not happen when $\G$ is a quasicircle.

\begin{lem}\label{lem:proximity}
Suppose that $\Omega$ is a Jordan domain whose boundary is  a $K$-quasicircle and that $\g_{\e}$ is a Jordan curve for some $\e>0$. Then there exists $M\geq 1$ depending only on $K$ such that for any $x \in \partial \Omega$, 
$\e \leq \dist(x,\g_{\e}) \leq M\e$.
\end{lem}

\begin{proof}
Suppose that $x\in \partial \Omega$ is a point having $\dist(x,\g_{\e})>\e$. By continuity, there exists a subarc $\Lambda$ of $\partial \Omega$ containing $x$ in its interior such that  $\dist(p,\g_{\e})=\dist(q,\g_{\e})=\e$ at 
its endpoints $p$ and $ q$ and that $\dist(y,\g_{\e})>\e$ at each interior point $y$ of $\Lambda$. Take $\tilde p, \tilde q \in \g_\e$ with $\dist(p,\g_{\e})=|p-\tilde p|$ and $\dist(q,\g_{\e})=|q-\tilde q|$, and let $\sigma$ be 
the component of $\g_\e \setminus \{\tilde p, \tilde q\}$ which together with $\Lambda \cup [p,\tilde p]\cup [q,\tilde q]$ encloses a Jordan domain contained in $\Omega\setminus \D_\e$. We further require that  choices of 
$\tilde p,\tilde q$ are made in such a way that $\sigma$ is minimal.

We claim that $\tilde p= \tilde q$. Otherwise, take $z$ to be an interior point of $\sigma$ and  take $z'$ to be a point in $ \G$ with $|z-z'|=\e$. By the non-crossing property, $z'$ must be in the interior of $\Lambda$; this 
contradicts to the definition of $\Lambda$. Therefore $\tilde p$ and $\tilde q$ must be the same, which implies that $|p-q|\leq 2\e$.

Since $\partial \Omega$ is a $K$-quasicircle, there exists $C>1$ depending only on $K$  so that
\[\dist(x, \g_\e) \leq |x-p| + |p-\tilde p| \leq \diam \Lambda +\e \leq C|p-q| +\e \leq (2C+1)\e.\qedhere \]
\end{proof}

For a Jordan domain $\Omega$ and a number $\e>0$, consider the open sets
\[ \D_{\e} = \{ x\in \Omega \colon \dist(x,\Omega) > \e\}. \]

In general $\D_{\e}$ need not be connected, and $\overline{\D_\e}$ and $\D_\e \cup \g_\e$ may not be the same. Nevertheless we have the following.

\begin{lem}[{\cite[Lemma 4.5]{VW}}]\label{lem:appr}
Let $\Omega$ be a Jordan domain and $\e > 0$. Then, every  connected component of $\D_{\e}$ is a Jordan domain.
\end{lem}

\begin{lem}[{\cite[Remark 4.13]{VW}}]\label{lem:boundary}
Suppose that $\Omega$ is a Jordan domain and that for some $\e > 0$, $\D_{\e} \neq \emptyset$,  $\g_{\e}\cup \D_{\e}$ is connected, and $\overline\D_{\e} \subsetneq \g_{\e}\cup \D_{\e} $. Then, there exists\ a component $D$ of 
$\D_{\e}$ and collinear points $x_0\in \partial D$ and $x_1,x_2 \in \G$  such that
\[ |x_0-x_1| = |x_0-x_2| = \e. \]
\end{lem}

\begin{lem}[{\cite[Remark 4.15]{VW}}]\label{lem:gammadelta}
Suppose that $\Omega$ is a Jordan domain and that for some $\e>0$, $\D_{\e}\neq\emptyset$ and $\g_{\e}\cup\D_{\e}$ is not connected. Then, there exist $\d \in (0,\e)$, a component $D$ of $\D_{\d}$, and collinear points 
$x_0 \in \partial D$ and $x_1,x_2 \in \G$ such that
\[ |x_0-x_1| = |x_0-x_2| = \d. \]
\end{lem}

\section{Quasispheres  over quasidisks}\label{sec:quasispheres}
We prove Theorem \ref{thm:quasisphere_t} in this section. For the proof, we need some well-known inequalities  from the classical function theory.

\subsection{Koebe estimates}
The \emph{Koebe 1/4 Theorem} states that \emph{if $f$ is a conformal map  from the unit disk $\B^2$ onto a simply connected domain $\Omega$, then}
\[|f''(0)/f'(0)| \le 2\]
\emph{and}
\[\dist (f(0),\partial \Omega)\geq |f'(0)| /4.\]
It follows by scaling that
for all $x\in \B^2$
\begin{equation}\label{eq:onequarter}
\frac{\dist(f(x),\partial \Omega)}{1 - |x|^2}
\leq |f'(x)|
\leq \frac{4\,\dist(f(x),\partial \Omega)}{1 - |x|^2}.
\end{equation}
In fact, there exists an absolute constant $A>1$ so that for all $x_0\in \mathbb B^2$ and
$x,  x_1, x_2\in B^2(x_0,(1-|x_0|)/2)\subset \B^2$
\begin{equation}\label{eq:onequarter-A}
\frac{1}{A}\,\frac{ |f(x_1)-f(x_2)|}{|x_1-x_2|}
 \leq |f'(x)|
 \leq A\,\frac{|f(x_1)-f(x_2)|}{|x_1-x_2|}
\end{equation}
and
\[ \dist(f(x_0),\partial \Omega)/A \leq \diam f(B^2(x_0,(1-|x_0|)/2))\leq A\, \dist(f(x_0),\partial \Omega). \]
See for example the books \cite{GarMar} and \cite{Pommerenke-1992}.

\bigskip

Set $D_x= \frac{\dist(f(x),\partial \Omega)}{1 - |x|} $. Then it follows from \eqref{eq:onequarter}, \eqref{eq:onequarter-A} and the Lipschitz continuity of the distance function $\dist(\cdot, \partial \Omega)$ that for all 
$x_0\in \mathbb B^2$ and $x_1, x_2\in B^2(x_0,(1-|x_0|)/2)\subset \B^2$
\begin{align}\label{eq:D-D}
|D_{x_1}-D_{x_2}| &=
\left |\frac{\dist(f(x_1),\partial \Omega)}{1 - |x_1|}-
\frac{\dist(f(x_2),\partial \Omega)}{1 - |x_2|}
\right | \nonumber
\\
&\leq \frac{(1 - |x_2|)|\dist(f(x_1),\partial \Omega)-\dist(f(x_2),\partial \Omega)|+|x_1-x_2|  \dist(f(x_2),\partial \Omega)}{(1 - |x_1|)(1 - |x_2|)}\nonumber
\\
&\leq \frac{|f(x_1)-f(x_2)|} {1 - |x_1|} + \frac{|x_1-x_2|}{1 - |x_1|}\frac{\dist(f(x_2),\partial \Omega)}{1 - |x_2|} \nonumber
\\
&\leq 8 A^2 \,\frac{|f'(x_0)||x_1-x_2|}{1 - |x_0|}.
\end{align}

\subsection{A class of $3$-dimensional quasiballs}
We  show that for any simply connected domain $\Omega$ in $\R^2$, the domain enclosed by the double-dome-like surface $\Sigma(\Omega,t)$ is quasiconformally equivalent to the unit ball $\B^3$.

\begin{thm}\label{thm:quasiball-cone}
Let  $\Omega$ be a simply connected domain in $\R^2$. Then the double cone  with base $\Omega$,
\begin{equation}\label{eq:coneK}
\mathcal K (\Omega,t) = \{(x,z) \colon    x\in {\Omega},\, |z|< \,\text{dist}(x, \partial \Omega)\},
\end{equation}
is $K_0$-quasiconformal homeomorphic to the unit ball $\B^3$,  for some absolute constant $K_0 > 1$.
\end{thm}

\begin{proof}
Define first  a  double circular cone $\mathcal C$ with base $\B^2$,
\begin{equation}\label{eq:coneC}
\mathcal{C}= \{(x,z)\in\mathbb{B}^2 \times\R \colon |z|<1-|x|\}.
\end{equation}
Fix a  conformal map $f$ from $\B^2$ onto $\Omega$ and define a homeomorphism $F \colon \mathcal C \to \mathcal K (\Omega,t)$ between cones, associated to $\Omega$ and $f$, by
\begin{equation}\label{eq:F}
F(x,z) = (f(x),  \frac{\dist(f(x),\partial \Omega)}{1 - |x|}\, z).
\end{equation}

We claim that $F$ is $K_1$-quasiconformal for some absolute constant $K_1 > 1$. The theorem follows from the claim and the fact that the cone $\mathcal C$ is quasiconformally equivalent to $\B^3$.

To this end, take $x_0\in \mathbb B^2$ and  $(x_1,z_1)$, $(x_2,z_2)$ in $\mathcal C$ with $x_1,x_2\in  B^2(x_0,(1-|x_0|)/2)\subset  \B^2$ and $-1 <z_1, z_2<1$. Inequalities \eqref{eq:onequarter}, \eqref{eq:onequarter-A} and 
\eqref{eq:D-D} yield
\begin{align*}
|F(x_1,z_1) - F(x_2,z_2)| &\leq |f(x_1) - f(x_2)| + |D_{x_1}z_1 - D_{x_2}z_2|\\
&\leq A |f'(x_0)||x_1-x_2| +  |D_{x_1}-D_{x_2}||z_2|+D_{x_1}|z_1-z_2| \\
&\leq A |f'(x_0)| |x_1-x_2|+ \frac{3}{2} |D_{x_1}-D_{x_2}|(1-|x_0|)\\
                         & \hspace{1.375in}  + 2 A^2 |f'(x_0)||z_1-z_2|\\
&\leq  15 A^2 |f'(x_0)||(x_1,z_1)-(x_2,z_2)|.
\end{align*}
For a lower estimate, note that when $|z_1-z_2| \leq  60 A^4 |x_1-x_2|$,
\begin{align*}
|F(x_1,z_1) - F(x_2,z_2)| &\geq |f(x_1)-f(x_2)|\\
                          &\geq \frac{1}{A}|f'(x_0)||x_1-x_2|\\
                          &\geq \frac{1}{61 A^5}|f'(x_0)||(x_1,z_1)-(x_2,z_2)|.
\end{align*}
On the other hand, if $|z_1-z_2| > 60 A^4 |x_1-x_2|$ then
\begin{align*}
|F(x_1,z_1) - F(x_2,z_2)|
&\geq |D_{x_1}z_1 - D_{x_2}z_2|\\
&\geq  D_{x_2}|z_1-z_2|-|(D_{x_1}-D_{x_2})z_1|\\
&\geq \frac{1}{4}|f'(x_2)||z_1-z_2|-8 A^2 \,\frac{|f'(x_0)||x_1-x_2|}{1 - |x_0|} |z_1|\\
& \geq \frac{1}{4 A^2}|f'(x_0)||z_1-z_2|-12 A^2 \,|f'(x_0)||x_1-x_2|\\
&\geq \frac{1}{30 A^2}|f'(x_0)||(x_1,z_1)-(x_2,z_2)|.
\end{align*}
From these estimates it follows that the restrictions of  $\frac{F}{|f'(x_0)|}$ on $\mathcal C \cap (B^2(x_0, (1-|x_0|)/2)\times \R)$ are uniformly bi-Lipschitz for all $x_0\in  \B^2$. Hence $F$ is $K_1$-quasiconformal for some 
absolute $K_1 >1 $.
\end{proof}

\subsection{Higher dimensional quasiballs}
The procedure of constructing quasiballs in $\R^3$ extends to all dimensions.

\begin{thm}\label{thm:quasiballs_n}
Let $\Omega$ be a simply connected domain in $\R^2$, set $\mathcal K^3(\Omega, t)= \mathcal K(\Omega, t)$  the cone  defined in \eqref{eq:coneK}, and write
\[ \mathcal K^n(\Omega, t)= \{(x,z) \colon    x\in {\mathcal K^{n-1}(\Omega, t)},\, |z|< \,\dist(x, \partial {\mathcal K^{n-1}(\Omega, t)})\} \]
for $n\ge 4$. Then for any $n\ge 3$, $\mathcal K^n(\Omega, t)$ is quasiconformally homeomorphic to the unit ball $\B^n$.
\end{thm}

We give an outline of the proof.

Set $\mathcal{C}^3= \mathcal{C}$, and
\[\mathcal{C}^n= \{(x,z)\colon x \in \mathcal{C}^{n-1}, |z|<\dist(x, \partial {\mathcal C^{n-1}})\}\]
for $n\ge 4$. The quasiconformality of the map $F\colon \mathcal C^3 \to \mathcal K^3(\Omega,t)$ was deduced from the Lipschitz property of a distance function and the  Koebe estimates \eqref{eq:onequarter} and 
\eqref{eq:onequarter-A}. These estimates  are essentially equivalent to the following
\begin{enumerate}
\item[(a)] for any $x_0 \in \mathbb{B}^2$, the restriction $f|_{\mathbb B^2(x_0, (1-|x_0|)/2)}$ is $C(x_0)$-bi-Lipschitz,
\item[(b)] the bi-Lipschitz constant $C(x_0)$  is comparable to  $\frac{\dist(f(x_0),\partial \Omega)}{\dist(x_0, \partial \B^2)}$ with ratios bounded above and below by constants  independent of $x_0$.
\end{enumerate}

The extended map  $F \colon \mathcal C^3 \to \mathcal K^3 (\Omega,t)$ defined in  \eqref{eq:F} inherits properties (a) and (b)  from that of $f$. Indeed, the proof of Theorem \ref{thm:quasiball-cone} implies that there exists 
an absolute constant $A_1 >1$ such that for any $(x_0,z_0) \in \mathcal C^3$ and $(x,z), (x_1,z_1), (x_2,z_2)\in B^3((x_0,z_0), \frac{1}{2}\dist((x_0,z_0), \partial \mathcal C^3))\subset \mathcal C^3$
\begin{equation}\label{eq:onequarter-A-F}
\frac{1}{A_1}\frac{ |F(x_1,z_1) - F(x_2,z_2)|}{|(x_1,z_1) - (x_2,z_2)|} \leq |f'(x)| \leq A_1 \frac{ |F(x_1,z_1) - F(x_2,z_2)|}{ |(x_1,z_1) - (x_2,z_2)|.}
\end{equation}
Since quasiconformal mappings map Whitney-type cubes to Whitney-type sets (\cite[Theorem 11]{Geh3}, \cite[p.93]{Heinonen}), we can deduce from \eqref{eq:onequarter-A-F} that
\begin{equation}\label{eq:onequarter-F}
\frac{1}{A_2}\frac{ \dist(F(x_0,z_0),\partial  \mathcal K^3(\Omega,t)) }{\dist((x_0,z_0), \partial \mathcal C^3))} \leq |f'(x_0)| \leq A_2
\frac{ \dist(F(x_0,z_0),\partial  \mathcal K^3(\Omega,t)) }{\dist((x_0,z_0),  \partial \mathcal C^3))}
\end{equation}
for some $A_2>1$ depending only on the dilatation $K_0$ in Theorem \ref{thm:quasiball-cone}.

Estimates \eqref{eq:onequarter-A-F} and \eqref{eq:onequarter-F}  and the Lipschitz property of a distance function allow the induction steps to continue. As a consequence, the map
\[ G\colon (x,z,z') \to (F(x,z), \, \frac{\dist(F(x,z),\partial \mathcal K^3(\Omega, t))}{\dist((x,z),\partial \mathcal C^3)}\, z') \]
for $(x,z) \in \mathcal C^3$ and $|z'|< \dist((x,z),\partial \mathcal C^3) $, is  quasiconformal  from $\mathcal C^4$ onto $\mathcal K^4(\Omega, t)$, and it satisfies the local bi-Lipschitz properties (a) and (b) needed for 
the next step. In particular, the restrictions of $\frac{G}{|f'(x_0)|}$ on
\[B^4((x_0,z_0,z_0'), \frac{1}{2}\dist((x_0,z_0,z_0'), \partial \mathcal C^4))\]
are uniformly bi-Lipschitz for all $(x_0,z_0,z_0')\in \mathcal C^4$. Since $\mathcal C^n$ is bi-Lipschitz equivalent to $\mathbb B^n$, $\mathcal K^n(\Omega, t)$ is quasiconformal to  $\B^n$ for any $n\ge 3$.

\subsection{Slit domains}
A domain $D \subset \R^n$ is called a \emph{slit domain} if $\R^n\setminus \overline{D} \subset \R^{n-1}\times\{0\}$. A theorem of Gehring on quasiconformal maps on slit domains in $\R^3$ states as follows.

\begin{thm}[{\cite[Theorem 5]{Geh2}}]\label{thm:Gehring-slit}
Suppose that $\Omega$ is a planar Jordan domain. Then the slit domain  $\R^3\setminus \overline{\Omega}$ is quasiconformally homeomorphic to $\R^3\setminus \overline{\mathbb B^3}$ if and only if $\Omega$ is a quasidisc.
\end{thm}

We state a simple observation on slit domains that is needed later.

\begin{lem}\label{lem:slit}
Let $D$ be a bounded domain in $\R^n$, and $h_1,\, h_2 \colon \R^n\to \R $ be $L$-Lipschitz functions satisfying $h_2\leq h_1$ in $D$ and $h_1=h_2 = 0$ in $\R^n\setminus D$. Then the domain
\[ \R^{n+1} \setminus \{ (x,z) \colon x \in D \text{, } h_2(x) \leq z \leq h_1(x) \}\]
is locally $(L+2)$-bi-Lipschitz to, therefore $K$-quasiconformal to, the slit domain $\R^{n+1} \setminus \overline D$, for some $K\geq 1$ depending only on $L$.
\end{lem}

\begin{proof}
Define $H \colon \R^{n+1} \setminus \overline{D} \rightarrow \R^{n+1} \setminus \{ (x,z)\colon x \in \overline{D} , h_2(x) \leq z \leq h_1(x) \}$ with
\begin{equation*}
H(x,z) =
\begin{cases}
 (x,z+h_1(x))  & \text{ if } x \in \overline{D} \text{  and  } z> 0, \\
 (x,z+h_2(x))  & \text{ if } x \in \overline{D} \text{  and  } z< 0, \\
 (x,z)         & \text{ if } x \not\in \overline{D}.
\end{cases}
\end{equation*}
It is straightforward to check that $H$ is $(L+2)$-bi-Lipschitz in each of the two half-spaces $\{(x,z)\colon x\in \R^n, z>0\}$ and $\{(x,z)\colon x\in \R^n, z<0\}$. Since $H$ is homeomorphic on $\R^{n+1} \setminus \overline{D}$, 
it is locally $(L+2)$-bi-Lipschitz and therefore $K$-quasiconformal for some $K$ depending only $L$.
\end{proof}

\subsection{Proof of Theorem \ref{thm:quasisphere_t}}\label{sec:proof_quasisphere_t}
For the sufficiency in  Theorem \ref{thm:quasisphere_t}, we apply Theorem \ref{thm:Gehring-Vaisala-characterization} of Gehring and \Vaisala. Note first from Theorem \ref{thm:quasiball-cone} that the bounded component of 
$\R^3\setminus \S(\Omega,t)$ is quasiconformal to  $\mathbb B^3$. On the other hand, by Lemma \ref{lem:slit}, the  unbounded component of $\R^3\setminus \S(\Omega,t)$ is quasiconformally homeomorphic to the slit domain 
$\R^3\setminus \overline{\Omega}$. Since $\Omega$ is a planar quasidisk, $\R^3\setminus \overline{\Omega}$ is quasiconformal to the exterior $\R^3\setminus \overline{\mathbb B^3}$ of the unit ball by Theorem \ref{thm:Gehring-slit}.
Therefore, $\S(\Omega,t)$ is a quasisphere by Theorem \ref{thm:Gehring-Vaisala-characterization}.

It remains to prove the necessity.

Suppose that the double-dome-like surface $\S(\Omega,t)$ is a $K$-quasisphere, hence a quasisymmetric sphere. Therefore  $\S(\Omega,t)$ is $\lambda-\text{LLC}$ for some $\lambda >1$ depending only on $K$. We claim that $\partial\Omega$ 
satisfies the $2$-point condition \eqref{eq:3pts} with $C= 32\lambda^2$.

Suppose the claim is false. Then there exist points $x_1,x_2,x_3,x_4 \in \partial\Omega$ in cyclic order such that,
\begin{equation}\label{eq:n=3eq}
|x_2-x_1|,|x_4-x_1| > 32\lambda^2|x_3-x_1|.
\end{equation}
Define
\[ d = \inf\{ \diam{\sigma} \colon \sigma \subset \overline \Omega \text{ is a continuum that contains }x_1\text{ and }x_3\}. \]
Since $x_1$ and $x_3$ are in the set $\S(\Omega,t) \cap B^3(x_1,2|x_1-x_3|)$, they are, by $\lambda-\text{LLC}_1$, contained in a continuum  $E$ in $ \S(\Omega,t)\cap B^3(x_1,2 \lambda |x_1-x_3|)$, and hence in the projection 
$\pi(E)$. Thus
\[d \leq \diam \pi(E)\leq \diam E \leq 4 \lambda |x_1-x_3|.\]
Next fix a continuum $\sigma$ in $\Omega$ that contains  $x_1$ and $x_3$ and has $\diam{\sigma} \leq 4\lambda|x_1-x_3|$. Since points $x_2$ and $x_4$ are in the set $ \S(\Omega,t) \setminus B^3(x_1,16\lambda^2|x_1-x_3|)$, they 
are, by the  $\lambda-\text{LLC}_2$ property, contained in a continuum  $F$ in $ \S(\Omega,t) \setminus  B^3(x_1,16\lambda|x_1-x_3|)$, and hence in the projection $\pi(F)$. Since $\pi(F)$  is a continuum in $\overline{\Omega}$, 
it intersects $\sigma$. Take a point $w\in F$ with $\pi(w) \in \sigma \cap \pi(F)$. Hence,
\begin{align*}
16\lambda|x_1-x_3|  < |w-x_1| &\leq |w-\pi(w)| + |\pi(w)-x_1| \\
                    &= \dist(\pi(w), \partial \Omega)+ |\pi(w)-x_1|\\
                    &\leq 2 |\pi(w)-x_1|\leq  8\lambda |x_1-x_3|,
\end{align*}
which is a contradiction. Therefore $\partial\Omega$ satisfies the $2$-point condition with $C = 32 \lambda^2$, hence it is a quasicircle.

This completes the proof of Theorem \ref{thm:quasisphere_t}.

\medskip

The proof in fact shows the following.

\begin{rem}\label{rem:QStoQcircle}
If the double-dome-like surface $\S(\Omega,t)$ is LLC then $\partial \Omega$ is a quasicircle.
\end{rem}

\subsection{Higher dimensional quasispheres}
The method of constructing quasispheres extends to all dimensions.

\begin{thm}\label{thm:quasispheres_n}
Suppose that $\Omega$ is a quasidisk in $\R^2$. Then $\partial(\mathcal K^n(\Omega, t))$ is a quasisphere in $\R^n$ for every $n\geq 3$.
\end{thm}

\begin{proof}
Since $\mathcal K^n(\Omega, t)$ is quasiconformal to $\mathbb B^n$ by Theorem \ref{thm:quasiballs_n}, it suffices to check that $\R^n \setminus \overline{\mathcal K^n(\Omega, t)}$ is quasiconformal to 
$\R^n \setminus \overline{\B^n}$ in view of  Theorem \ref{thm:Gehring-Vaisala-characterization}.

Fix, by Theorem \ref{thm:quasisphere_t}, a quasiconformal $g\colon \R^3 \to \R^3$  that maps $\mathcal C^3  $ onto $\mathcal K^3(\Omega,t)$. Extend $g$, by a quasiconformal extension theorem of Tukia and {\Vaisala} 
\cite[Theorem 3.12]{TukVais}, to a quasiconformal homeomorphism $G\colon \R^4 \to \R^4$ with $G|_{\R^3\times \{0\}}=g$;  hence $\R^4 \setminus \overline{\mathcal C^3 } $ is quasiconformal to 
$\R^4\setminus \overline{\mathcal K^3(\Omega,t)}$. Furthermore, by Lemma \ref{lem:slit}, $\R^4\setminus \overline{\mathcal K^3(\Omega,t)}$ is quasiconformal to $\R^4\setminus \overline{\mathcal K^4(\Omega,t)}$; and 
$\R^4 \setminus \overline{\mathcal C^3 } $ is quasiconformal to $\R^4 \setminus \overline{\mathcal C^4 } $, hence quasiconformal to $\R^4 \setminus \overline{\mathbb B^4 } $. Therefore 
$\R^4 \setminus \overline{\mathcal K^4(\Omega, t)}$ is quasiconformal to $\R^4 \setminus \overline{\B^4}$. This procedure can be continued inductively to all dimensions. We conclude that 
$\R^n \setminus \overline{\mathcal K^n(\Omega, t)}$ is quasiconformal to $\R^n \setminus \overline{\B^n}$ for any $n\geq 3$.
\end{proof}

\begin{rem}
Theorem \ref{thm:quasispheres_n} remains true if we replace $\varphi(t) =t$ with gauges $\varphi \in \mathcal F$ that are bi-Lipschitz on $[0,\infty)$.

Indeed, suppose that $\Omega$ is a $K$-quasidisk and $\varphi_1, \varphi_2, \dots $ are $L_i$-bi-Lipschitz gauges in $ \mathcal{F}$. Set
\[ \mathcal{K}^3 = \{ (x,z) \colon x \in \Omega, \, |z| < \varphi_1(\dist(x,\partial\Omega))\},\]
and
\[ \mathcal{K}^{n+2} = \{ (x,z) \colon x \in \mathcal{K}^{n+1}, \, |z| < \varphi_n(\dist(x,\partial\mathcal{K}^{n+1}))\}\]
for $n \geq 2$. Then, for any $n\ge 1$, the surface $\partial\mathcal{K}^{n+2}$ is a  $K_n$-quasisphere in $\R^{n+2}$, with $K_n$ depending only on $n, K, L_1,\dots,L_n$.
\end{rem}

\section{Linear local connectivity and the level quasicircle property}\label{sec:LLC+LQC}
In this section we establish a relation between  the LLC property of $\S(\Omega,\varphi)$ and the LQC property on $\Omega$.

\begin{prop}\label{prop:LLCiffLQC}
Let $\Omega$ be a Jordan domain. Then the surface $\S(\Omega,\varphi)$ is $\text{LLC}$  for all $\varphi\in\mathcal{F}$ if and only if  $\Omega$ has the LQC property.
\end{prop}

We need a stronger form of $\text{LLC}_2$ for planar quasicircles for  Lemma \ref{lem:LLC1}; the   straightforward proof is left to the reader.

\begin{rem}\label{rem:strongllc}
Let  $\G\subset\R^2$ be a $K$-quasicircle. Then there exists $\lambda=\lambda(K) >1$ such that for any $x\in\R^2$, any $r>0$, and any two points $y_1, y_2 \in \G \setminus \overline{B}^2(x,r)$, there exists a subarc in 
$\G \setminus B^2(x,r/\lambda)$ that contains $y_1$ and $y_2$.
\end{rem}

From here onward, given $\Omega$ and $\varphi$ we set
\[\S(\Omega,\varphi)^+=   \S(\Omega,\varphi)\cap\R^3_+ \,\,\text{ and }\,\, \S(\Omega,\varphi)^-= \S(\Omega,\varphi)\cap\R^3_- ,\]
and for a given subset  $S$ of  $\overline{\Omega}$, let
\[ S^+=\{(x,\varphi(\dist(x,\partial \Omega))\colon x\in S \}\]
and
\[ S^-=\{(x,-\varphi(\dist(x,\partial \Omega))\colon x\in S \},\]
be the lifts of  $S$  to $\S(\Omega,\varphi)^+$ and $\S(\Omega,\varphi)^-$, respectively. For instance, $\D_{\e}^-$ is the part of surface in $\S(\Omega,\varphi)^-$  whose projection on $\R^2\times\{0\}$ is $\D_{\e}$, and 
$w^+=(w,\varphi(\dist(w,\partial \Omega)))$ when $w$ is a point in $\Omega$.

\subsection{Sufficient conditions for $\S(\Omega,\varphi)$ to be $\text{LLC}$}\label{sec:LLCsuf}

\begin{lem}\label{lem:LLC1}
Let $\Omega$ be a Jordan domain that has the level quasicircle property ($\text{LQC}$) and $\varphi$ be a homeomorphism in $\mathcal F$. In particular, there exist $K, L, M >1$ and $\e_0>0$ so that $\g_{\e}$ is a $K$-quasicircle 
for every $\e \in [0,\e_0]$ and that  $\varphi$ is $L$-Lipschitz in $[\e_0/3,\infty)$ and satisfies
\[ \varphi(t) >M t \text{  for all }  t \in [0,\diam \G]. \]
Then, $\S(\Omega,\varphi)$ is $\lambda-\text{LLC}$ for some $\lambda> 1$ depending only on  $K$, $L$, $M$, $\e_0$ and $\diam{\Omega}$.
\end{lem}

\begin{proof}
Since all level curves $\{\gamma_{\epsilon}\}_{0\leq \epsilon <\e_0}$ are $K$-quasicircles, they satisfy the $2$-point condition \eqref{eq:3pts} with  a common $C > 1$.

\bigskip

\noindent{\bf  Part I.} The $\text{LLC}_1$ property.

We claim that there exists  $\lambda >1$ such that for any $y_1$ and $y_2$ in $\S(\Omega,\varphi)$, there exists a curve $\sigma$ in $\S(\Omega,\varphi)$ joining $y_1,y_2$ and having  $\diam{\sigma} \leq \lambda|y_1-y_2|$. This 
statement implies that $\S(\Omega,\varphi)$ is $(1+2\lambda)-\text{LLC}_1$. For the proof of the claim we consider three cases.

\emph{Case 1}. Suppose that both $y_1$ and $y_2$ are in $ \overline{\D_{\e_0}^+}$, or both are in $\overline{\D_{\e_0}^-}$. Since $\D_{\e_0}$ is a $K$-quasidisk, $\overline{\D_{\e_0}}$  contains an arc $\tau$ joining $\pi(y_1)$ 
to $\pi(y_2)$ whose diameter is at most $C  |\pi(y_1)-\pi(y_2)|$. Because $\varphi$ is $L$-Lipschitz on $[\e_0/3,\infty)$, $\tau$ lifts to a curve $\sigma=\tau^+$ on $\S(\Omega,\varphi)$ which connects $y_1$ to $y_2$ and has 
$\diam{\sigma}\leq (C+L) |\pi(y_1)-\pi(y_2)|$.

\emph{Case 2.} Suppose that both $y_1$ and $y_2$ are in the same half-space and at least one of them is not in $\overline{\D_{\e_0}^+}\cup\overline{\D_{\e_0}^-}$. Assume, for instance, that 
$y_1,y_2\in\R^3_+$, $y_1\in \S(\Omega,\varphi)^+\setminus \D_{\e_0}^+$ and  $\pi(y_1) \in \gamma_{\epsilon_1}$ and $\pi(y_2) \in \gamma_{\epsilon_2}$ with  $\epsilon_1 \leq \epsilon_2$. Other subcases can be treated analogously.

Take $w$ in $\S(\Omega,\varphi)^+$ such that $\pi(w)$ is  a point on $\gamma_{\epsilon_1}$ that is nearest to $\pi(y_2)$, thus $|\pi(y_2)-\pi(w)| = \epsilon_2 - \epsilon_1$.

Since $\varphi$ is increasing in $[\e_1,\e_2]$ we have that
\begin{align*}
|y_2-w| &\leq |\pi(y_2)-\pi(w)| + \varphi(\e_2) - \varphi(\e_1)\\
&\leq |\pi(y_2)-\pi(y_1)| + \varphi(\e_2) - \varphi(\e_1)\\
&\leq 2|y_1-y_2|
\end{align*}
and
\begin{align*}
|y_1-w| &= |\pi(y_1) - \pi(w)|\\
&\leq |\pi(y_1)-\pi(y_2)| + |\pi(y_2) - \pi(w)|\\
&\leq 2|\pi(y_1) - \pi(y_2)|\\
&\leq 2|y_1-y_2|.
\end{align*}
Since $\g_{\e_1}$ satisfies the $2$-point condition \eqref{eq:3pts}, there exists a subarc $\tau \subset \g_{\e_1}$ joining $\pi(w)$ to $\pi(y_1)$ and having
\[ \diam{\tau_1} \leq C|\pi(w) - \pi(y_1)| \leq C|w-y_1| \leq 2C|y_1-y_2|.\]
Let $\sigma_1$ be the lift of $\tau$ and $\sigma_2$ be the lift of the line segment $[\pi(w),\pi(y_2)]$ on $\S(\Omega,\varphi)^+$ respectively. Since $\varphi$ is increasing,
\[ \diam{\sigma_2} \leq \varphi(\e_2)-\varphi(\e_1) + \e_2 - \e_1 \leq 2|y_2-w| \leq 4|y_1-y_2|.\]
Then $\sigma = \sigma_1\cup\sigma_2$ is an arc in $\S(\Omega,\varphi)$ that connects $y_1$ to $y_2$ and has $\diam{\sigma} \leq (2C+4)|y_1-y_2|$.

\emph{Case 3.} Suppose that $y_1$ and $y_2$ are in two different half-spaces. Consider, for instance, that $y_1 \in\S(\Omega,\varphi)^+$ and $y_2 \in\S(\Omega,\varphi)^-$ and that $\pi(y_1) \in \g_{\e_1}$ and 
$\pi(y_2)\in \g_{\e_2}$. Then  $|y_1 - y_2| \geq \varphi(\epsilon_1) + \varphi(\epsilon_2)$.

Take $v_1,v_2 \in \G$ such that $|\pi(y_1) - v_1| = \epsilon_1$ and $|\pi(y_2) - v_2| = \epsilon_2$. Denote by $\sigma_1$ the lift of the segment $[\pi(y_1),v_1]$ on $\S(\Omega,\varphi)^+$ and by $\sigma_2$ the lift of the 
segment $[\pi(y_2),v_2]$ on $\S(\Omega,\varphi)^-$, respectively. The assumptions on $\varphi$ in the statement of the lemma yield
\[ \diam{\sigma_i} \leq |y_i-v_i| \leq \e_i+ \varphi(\e_i) \leq (1+1/M)\varphi(\e_1) \leq (1+1/M)|y_1-y_2|.\]
for $i=1$ and $2$. Let $\sigma_3$ be a subarc of $\G$ joining $v_1,v_2$ and having $\diam{\sigma_3} \leq C|v_1-v_2|$. It follows that
\[ \diam{\sigma_3} \leq C(|\pi(y_1)-\pi(y_2)| + \e_1 + \e_2) \leq C(2+1/M)|y_1-y_2|.\]

The path $\sigma = \sigma_1\cup\sigma_2\cup\sigma_3$ joins $y_1,y_2$ in $\S(\Omega,\varphi)$ and has
\[\diam{\sigma} \leq (2+2/M)(C+1) |y_1-y_2|.\]
This proves the claim and  Part I.

\bigskip

\noindent{\bf Part II.}The $\text{LLC}_2$ property.

We claim that there exists  $\lambda>1$ such that, for any  $x \in \S(\Omega,\varphi)$, $r>0$ and $y_1,y_2\in \S(\Omega,\varphi)\setminus B^3(x,r)$, there exists a continuum 
$E \subset \S(\Omega,\varphi)\setminus B^3(x,r/\lambda)$ that contains $y_1$ and $y_2$. Since $\S(\Omega,\varphi)\setminus B^3(x,r)$ is nonempty,
\[ r < \diam{\S(\Omega,\varphi)} \leq \diam{\G} + 2\varphi(\diam{\G}). \]
To verify the claim, it suffices to show the existence of $\lambda_0 >1$ when
\[ 0< r \leq r_0 = \min \{ \e_0/3 , \varphi(\e_0/3) \}. \]
For then, when $r> r_0$, points $y_1$ and $y_2$ are in $\S(\Omega,\varphi)\setminus B^3(x,r_0)$, and therefore are contained in a continuum $E$ in   $\S(\Omega,\varphi)\setminus B^3(x,r_0/\lambda_0)$ which in turn is contained 
in $ \S(\Omega,\varphi)\setminus B^3(x,r/(C_0 \lambda_0))$, where
\[C_0=\frac{\diam{\G} + 2\varphi(\diam{\G})}{r_0}.\]

We observe that  projection $\pi(z)$ of a point $z$ in $ \S(\Omega,\varphi)\setminus B^3(x,r)$ need not be in  $\overline{\Omega} \setminus B^2(\pi(x),r) $; this fact complicates the argument below.

Since $0< r<\e_0/3$, the following two cases suffice.

\emph{Case 1}. Suppose that $ \S(\Omega,\varphi) \cap B^3(x,r)   \subset \overline{\D_{\e_0/3}^+} \cup \overline{\D_{\e_0/3}^-}$. Because $\varphi$ is increasing and $0< r \leq \varphi(\e_0/3)$, $B^3(x,r)$ intersects only one of 
$\overline{\D_{\e_0/3}^+}$ and $\overline{\D_{\e_0/3}^-}$; assume that $\S(\Omega,\varphi) \cap B^3(x,r) \subset \overline{\D_{\e_0/3}^+}$. Since $\varphi$ is $L$-Lipschitz in $[\e_0/3,\infty)$, $ \S(\Omega,\varphi)\cap B^3(x,r)$ 
contains lift
\[G=\{(w, \varphi(\dist(w,\G)))\colon y\in B^2(\pi(x),c(L)r)\}\]
of the disk $B^2(\pi(x),c(L)r) $ on $\overline{\D_{\e_0/3}^+}$, for some $0<c(L)<1$ independent of the point $x$. Then $ \S(\Omega,\varphi) \setminus G$ is a continuum in $\S(\Omega,\varphi)\setminus B^3(x,c(L)r)$ that contains 
$y_1,y_2$.

\emph{Case 2}. Suppose that $B^3(x,r) \cap (\D_{\e_0}^+ \cup \D_{\e_0}^-) = \emptyset$. Given $y_1,y_2\in \S(\Omega,\varphi)\setminus B^3(x,r)$, we  consider the following subcases.

(i)\, Both $y_1,y_2 \in \D_{\e_0}^+$ or both $y_1,y_2 \in \D_{\e_0}^-$. There is nothing to prove because either $\D_{\e_0}^+$ or $\D_{\e_0}^-$ is a continuum exterior to $ B^3(x,r)$.

(ii)\, Both $y_1,y_2 \in \S(\Omega,\varphi) \setminus (\D_{\e_0}^+\cup \D_{\e_0}^-)$.  We assume further that $y_1\in \S(\Omega,\varphi)^+ \setminus \D_{\e_0}^+$ and $y_2\in \S(\Omega,\varphi)^- \setminus \D_{\e_0}^-$; other 
possibilities can be treated similarly. Assume also that $\pi(y_1)$ is on  $\g_{\e_1}$ and $\pi(y_2)$ is on $\g_{\e_2}$ for some $\e_1,\e_2 \in [0,\e_0]$.

Fix as we may a half-line in $\R^2$ starting at a point $x_0 \in \D_{\e_0}$ that does not intersect  $B^2(\pi(x),r)$. Let $v$ be the point on $l\cap \G$ closest to $x_0$ and $v_i$ be the point on $l \cap \g_{\e_i}$ closest to 
$x_0$ for $i= 1$ or $2$. Note that  $v,v_1,v_2$ are contained in a line segment entirely in $\overline \Omega$. Denote by $w_1 = (v_1,\varphi(\e_1))$ and $w_2 = (v_2,-\varphi(\e_2))$ the lifts of $v_1$ and $v_2$ on 
$\S(\Omega,\varphi)^+$  and $ \S(\Omega,\varphi)^-$, respectively.

Let $H$ be the hyperplane $\R^2 \times \{\varphi(\e_1)\}$ in $\R^3$. Then $\tau = \S(\Omega,\varphi)\cap H$, which is the lift $\g_{\e_1}^+$, is a $K$-quasicircle on $H$ that contains $y_1$ and $w_1$. Assume for a moment that 
$B^3(x,r)\cap H$ is nonempty, so it is a disk $B^2(\pi(x), \rho)\times \{\varphi(\e_1)\}$ with $0< \rho \leq r$. By Remark \ref{rem:strongllc}, there exists $\lambda_1>1$ depending only on $K$ and a subarc $\sigma_1$ of 
$\tau \setminus (B^2(x_1, \rho/\lambda_1)\times \{\varphi(\e_1)\})$ that connects $y_1$ to $w_1$. By elementary geometry, the curve $\sigma_1$ is in fact contained in  $\S(\Omega,\varphi)\setminus B^3(x,r/\lambda_1)$. When 
$B^3(x,r)\cap H$ is empty, a curve  $\sigma_1$ with the required properties exists trivially. Similarly, $y_2,w_2$ can be joined by an arc  $\sigma_2 \subset \S(\Omega,\varphi)\setminus B^3(x,r/\lambda_1)$ on the plane 
$\R^2 \times \{-\varphi(\e_2)\}$. Finally the union $\sigma_3 = {[v,v_1]}^+ \cup {[v,v_2]}^- $ of the lifts, of segments $[v,v_1]$ and $[v,v_2]$, connects $w_1$ to $v$ then to $w_2$ in $\S(\Omega,\varphi)$. Then, $y_1,y_2$ can 
be joined in $\S(\Omega,\varphi) \setminus B^3(x,r/\lambda_2)$ by the curve $\sigma = \sigma_1\cup\sigma_2\cup \sigma_3$.

(iii)\, $y_1\in \S(\Omega,\varphi) \setminus (\D_{\e_0}^+\cup \D_{\e_0}^-)$ and $y_2 \in \D_{\e_0}^+ \cup \D_{\e_0}^-$. We consider  only  the case when  $y_1\in \S(\Omega,\varphi)^+ \setminus \D_{\e_0}^+$ and 
$y_2 \in \D_{\e_0}^-$. Proceed as in (ii) to obtain a  half-line $l$, points $x_0$, $v, v_1$ and  $w_1$, and a curve $\sigma_1$ connecting $y_1$ to $w_1$. We next take $v_0$ to be the point on $l \cap \g_{\e_0}$ closest to $x_0$ 
and $w_0=(v_0, -\varphi(\e_0))$ be its lift on $\S_(\Omega,\varphi)^-$. Join $y_2$ to $w_0$ by any curve $\sigma_2'$ contained in $\overline {\D_{\e_0}^-}$. Then proceed as in subcase (ii) to define a curve $\sigma_3'$ in 
$\S(\Omega,\varphi)$ connecting $w_1$ to $v$ then to $w_0$ and having projection $\pi(\sigma_3')$ in $ l$. The curve $\sigma_1\cup\sigma_2'\cup \sigma_3'$ has the desired properties.

(iv)\, Either $y_1\in \D_{\e_0}^+$ and $y_2 \in \D_{\e_0}^-$, or $y_1\in \D_{\e_0}^-$ and $y_2 \in \D_{\e_0}^+$. We assume the former and fix a half-line $l \subset \R^2$, points $v_0$ $w_0$ and a curve $\sigma_2'$ as in (iii).
Let $u_0=(u_0, \varphi(\e_0))$ be the lift of $v_0$ on $\S(\Omega,\varphi)^+$, and let  $\sigma_1'$ be any curve in $\overline {\D_{\e_0}^+}$ joining $y_1$ to $u_0$. Finally as in the previous cases, choose a curve 
$\sigma_3''  \subset \S(\Omega,\varphi)$  joining $w_0$ to $v$ then to  $u_0$ and having projection $\pi(\sigma_3'')$ in $ l$. The curve $\sigma_1'\cup\sigma_2'\cup \sigma_3''$ has the properties required.
\end{proof}

\subsection{Necessary conditions for $\S(\Omega,\varphi)$ to be  $\text{LLC}$}\label{sec:LLCnec}
We prove the necessity in Proposition \ref{prop:LLCiffLQC}.

\begin{lem}\label{lem:LLCtoLQC}
Suppose that $\Omega$ is a Jordan domain and that $\S(\Omega,\varphi)$ is $\text{LLC}$ for every $\varphi \in \mathcal{F}$. Then $\Omega$ has the LQC property.
\end{lem}

We show first in Lemma \ref{lem:LJC} that if $\S(\Omega,\varphi)$ is $\lambda-\text{LLC}_1$ for a gauge $\varphi$ having rapid growth near $0$, then $\g_{\e}$ has the level Jordan curve property. We show next in Lemma 
\ref{lem:lqc1} that if $\G$ is a quasicircle and $\S(\Omega,\varphi)$ is $\lambda-\text{LLC}_1$ for a gauge $\varphi$ having rapid growth near $0$, then $\Omega$ satisfies the LQC property. Since $\S(\Omega,t)$ is LLC, $\G$ 
must be a quasicircle  by Remark \ref{rem:QStoQcircle}. This completes the proof of Lemma \ref{lem:LLCtoLQC}.

\begin{lem}\label{lem:LJC}
Let $\Omega$ be a Jordan domain and  $\varphi$ be a function in $\mathcal{F}$ whose almost everywhere derivative $\varphi'(t)\to \infty$ as $t\to 0$. Suppose that $\S(\Omega,\varphi)$ is $\lambda-\text{LLC}_1$. Then there exists 
$\epsilon_0 > 0$ depending only on $\lambda,\varphi$ such that the set $\g_{\e}$ is Jordan curve for every $0\leq \epsilon \leq \epsilon_0$.
\end{lem}

\begin{proof}
It follows from the assumption on $\varphi'$ that there exists $\e_0 >0 $ so that
\begin{equation}\label{eq:slope}
\varphi(t_2) - \varphi(t_1) > 6\lambda(t_2-t_1) \text{ for any } 0 < t_1 \leq t_2 \leq \e_0.
\end{equation}

The proof contains two claims. First we prove that the set $\D_{\e}$ is connected for any  $0< \e < \e_0$, which implies by Lemma \ref{lem:appr} that $\D_{\e}$ is a Jordan domain. Next we show that $\g_{\e}=\partial\D_{\e}$ 
whenever  $0<\e < \e_0$,  which, combined with the first claim, states that $\g_{\e}$ is a Jordan curve when $\e < \e_0$.

\medskip

\emph{Step 1.} We claim that  the open set $\D_{\e}$ is connected for any $0<\e < \e_0$. Suppose, to the contrary, that there exists $0<\e<\e_0$ such that the open set $\D_{\e}$ has at least two components $D_1$ and $D_2$. This 
would imply, by the continuity of the distance function, that $\D_{\e'}\cup \g_{\e'}$ is not connected for some $\e' \in (\e,\e_0)$. From Lemma \ref{lem:gammadelta} it follows that there exist $\d < \e_0$, a component $D$ of 
$\D_{\d}$, and three distinct collinear points $x_0 \in \partial D$ and $x_1,x_2 \in \G$ such that
\[ |x_0 - x_1| = |x_0 - x_2| = \d, \]
and from Lemma \ref{lem:appr} it follows that  $D$ is a Jordan domain. Observe that $D$ is exterior to the disks $\overline{B^2}(x_1,\d)$ and $ \overline{B^2}(x_2,\d)$ and that  $\G \cap B^2(x_0,\d) =\emptyset$.

Fix a simple arc  $\sigma$ in $ D \cup \{x_0\}$ with end points  $x_0$ and $w_0$. Heuristically, the set
\[ W=\{(x,z)\colon x \in \sigma, |z|\leq \varphi(\dist(x, \G))\} \]
serves as a tall, wide wall which prevents two nearby points $z_1^+$ and $z_2^+$ on two sides of $W$ to be joined by a short path without piercing  $W$. This violates the $\text{LLC}_1$ property.

To this end, we fix a point $y_0$  on $ \sigma$ satisfying
\[ |y_0 - x_0| < \min \left \{ \frac{|w_0 - x_0|}{2}, \frac{\d}{16\lambda} \right \}, \]
and set $r = |y_0-x_0|$. Simple geometric consideration shows that
\[ \dist(y_0,\overline{B^2}(x_i,\d)) \leq \frac{r^2}{\d} < \frac{r}{16\lambda}. \]
Set $\rho = \d - \frac{r}{8\lambda}$. Since $\dist(y_0,\G) > \d$, by the continuity of the distance function, we can find points $ z_1 \in [y_0,x_1]$ and $ z_2 \in [y_0,x_2]$ with $\dist(z_1,\G)=\dist(z_2,\G)=\rho$. Since 
$|y_0-z_i|\geq \dist(y_0, \G) - \dist(z_i, \G)>\d - \rho= \frac{r}{8\lambda}$, the point $z_i$ is contained in  $B^2(x_i,\d)$, thus
\[\rho \leq |z_i-x_i| < \d\]
for $i=1,2$. Therefore $z_1$ and $z_2$ are in two different components of $B^2(y_0,r)\setminus \sigma$, and for $i=1,2$
\[|z_i-y_0|= |y_0-x_i| - |x_i-z_i| < \d+ \frac{r}{16\lambda}- \rho =\frac{3 r}{16\lambda}.\]

Recall that ${z}_1^+$ and ${z}_2^+$ are the lifts of $z_1$ and $z_2$ on $\S(\Omega,\varphi)^+$. Since
\begin{align*}
|z_1^+ - z_2^+| &=|z_1 -z_2| \leq  |y_0-z_1| + |y_0 -z_2| < \frac{3 r}{8\lambda},
\end{align*}
points $z_1^+$ and $z_2^+$ are in the ball
\[ B = B^3\left ({z}_1^+,\frac{3r}{8\lambda}\right ). \]
Since $\S(\Omega,\varphi)$ is $\lambda-\text{LLC}_1$, the points ${z}_1^+,{z}_2^+$ are contained in a continuum $E$ in $\lambda B \cap \S(\Omega,\varphi)$, where $\lambda B= B^3\left ({z}_1^+,\frac{3r}{8}\right )$. Note that 
if $u \in \pi(\lambda B)$ then
\[ |u-y_0| \leq |u-z_1| + |z_1-y_0| \leq \frac{3r}{8} + \frac{3r}{16\lambda}  < r, \]
which shows that $\pi(E) \subset \pi(\lambda B) \subset B^2(y_0,r)$.

Note that for any  $w \in E$, $w= (\pi(w), \varphi(\dist(\pi(w),\G)))$ and
\[| \varphi(\dist(\pi(w),\G)) - \varphi (\rho) | = |\varphi(\dist(\pi(w),\G)) - \varphi(\dist(z_1,\G)| \leq  |w - z_1^+| \leq \frac{3r}{8},\]
and on the other hand by \eqref{eq:slope}
\begin{equation*}
| \varphi(\dist(\pi(w),\G)) - \varphi (\rho) |\geq 6 \lambda | \dist(\pi(w),\G) - \rho|.
\end{equation*}
From these estimates it follows that $\dist(\pi(w),\G) < \rho + \frac{r}{16}< \d$ for any $w\in E$, and as a consequence $\pi(E)$ does not intersect $\sigma$. This leads to a contradiction, because $\pi(E)$ is a continuum 
containing two points $z_1$ and $z_2$ lying in two separate components of $B^2(y_0,r)\setminus \sigma$.

\medskip

\emph{Step 2.} We claim that $\partial\D_{\e} = \g_{\e}$ for each $\e \in (0, \e_0)$. Suppose the contrary. Pick a number $\e < \e_0$ for which $\partial\D_{\e} \subsetneqq \g_{\e}$. Then, by Lemma  \ref{lem:gammadelta}, there 
exists a component $D$ of $\D_{\e}$ and collinear points $x_0 \in \partial D$ and $x_1,x_2 \in \G$ such that
\[ |x_0 - x_1| = |x_0 - x_2| = \e. \]
The argument leading to a contradiction is similar to that  in Step 1; we omit the details.
\end{proof}

\begin{lem}\label{lem:lqc1}
Let $\Omega$ be a Jordan domain interior to a  $K$-quasicircle $\G$, and $\varphi$ be a function in $\mathcal{F}$ whose almost everywhere derivative  $\varphi'(t)\to \infty$ as $t\to 0$. If $\S(\Omega,\varphi)$ is 
$\lambda-\text{LLC}_1$ for some $\lambda >1$ then $\Omega$ has the LQC property. In particular, there exist $\epsilon_0 > 0$ depending only on $\lambda,\varphi$ and $K'>1$ depending only on $K, \lambda$ such that $\g_{\e}$ is 
a $K'$-quasicircle for any $0< \e \leq \e_0$.
\end{lem}

\begin{proof}
Fix $\e_0>0$ so that
\begin{equation}\label{eq:slope2}
\varphi'(t)  \geq 10\lambda \quad \text{ for almost every  } t\in(0,\e_0).
\end{equation}
By Lemma \ref{lem:LJC}, the level set $\g_{\e}$ is a Jordan curve for any $\e < \e_0$.

Since $\G$ is a $K$-quasicircle, there exists $C>1$ depending  only on $K$ so that
\begin{equation}\label{eq:2pt_lemma}
\diam{\G(x,y)} \leq C|x-y|,\,\,\, \text{ for all }  x,y \in \G.
\end{equation}
It suffices to show that  for every $\e \in (0,\e_0)$, the curve $\g_{\e}$ satisfies the $2$-point condition
\begin{equation}\label{eq:2pt_50C}
 \diam{\g_{\e}(x,y)} \leq 50\lambda C|x-y|,\,\,\, \text{ for all }  x,y \in \g_{\e}.
 \end{equation}
Otherwise, there exist $\e \in (0,\e_0)$ and points $x_1,x_2 ,x_3,x_4 $ on $\g_{\e}$ in cyclic order such that $x_2$ and $x_4$ are on two different components of $\g_{\e} \setminus \{x_1,x_3\}$ and that
\[ |x_2-x_1|, |x_4-x_1| > 25\lambda C|x_1-x_3|. \]

We claim that
\[ |x_1-x_3|< \frac{\e}{6\lambda}. \]
For each $i$ fix a point $p_i$ on $ \G$ that is nearest to $x_i$; so $|p_i-x_i| = \e$. By Lemma \ref{lem:orientation}, the points $p_1$, $p_2$, $p_3$, $p_4$ follow the ordering of $x_1$, $x_2$, $x_3$, $x_4$. However some of 
the points $p_i$ might coincide. Note that $|p_1-p_3| \leq 2\e + |x_1-x_3|$ and from \eqref{eq:2pt_lemma}  that
\begin{align*}
C|p_1-p_3|   &\geq \min\{|p_1 - p_2|, |p_1-p_4| \} \\
             &\geq \min\{|x_1 - x_2|, |x_1-x_4| \} - 2\e \\
             &\geq 25\lambda C|x_1-x_3| -2\e.
\end{align*}
Hence, $|x_1-x_3| < \frac{\e}{6\lambda}$. Set $d = |x_1-x_3|$.

Since $p_i$ is a nearest point  on $\G$ to $x_i$, the intersection $[x_i,p_i]\cap \gamma_{\epsilon-d}$ contains a single point $z_i$; and let ${z_i}^+ = (z_i,\varphi(\epsilon-d))$ be its lift on $\S(\Omega,\varphi)^+$. As before, 
consider a  wall
\[ W'=\{(x,z)\colon x \in \overline{\D_\e}, |z|\leq \varphi(\dist(x, \G))\}, \]
and observe that points $z_1^+$ and $z_2^+$ are separated by a thin part of wall $W'$. Any path joining $z_1^+$ to $z_2^+$ without piercing $W'$ must climb above or go around the wall; therefore such a path has to be long 
relative to the distance between $z_1^+$ and $z_2^+$.

Consider the ball $B = B^3({z}_1^+,5d)$. Since $|z_1-z_3| \leq |z_1-x_1| + |x_1-x_3| +|x_3-z_3| \leq 3d$ we have that ${z}_3^+ \in B$. By the $\lambda-\text{LLC}_1$ there is a continuum $E$ in $\lambda B\cap \S(\Omega,\varphi)$ 
that contains ${z}_1^+,{z}_3^+$, where $\lambda B = B^3({z}_1^+,5 \lambda d)$. Note  that for any $w\in E$, $w= (\pi(w), \varphi(\dist(\pi(w),\G)))$ and that
\[ 10 \lambda |\pi(w)-z_1| \leq |\varphi(\dist(\pi(w),\G))-\varphi(\dist(z_1,\G))|\leq |w - z_1^+|<5\lambda d. \]
Hence  $\pi(E)$ is contained in the annular region $A= \overline{\D}_{\e-2d} \setminus \D_{\e}$ bordered by two Jordan curves $\g_{\e-2d}$ and $\g_{\e}$. Since $\pi(E)$ is a continuum in $\pi(\lambda B)\cap\overline{\Omega}$ 
that contains $z_1,z_3$, it must intersect both components in $A \setminus ([x_2,p_2] \cup [x_4, p_4])$, hence  intersects at least one of the two segments $[x_2,p_2]$ and $[x_4,p_4]$. From this it follows that
\begin{align*}
\diam(\pi(E)) &\geq  \min\{|x_1 - x_2|, |x_1-x_4| \} - |x_1-z_1| -|x_3-z_3|\\
              &\geq 25\lambda C d - 12 d \\
              &>13 \lambda d.
\end{align*}
On the other hand, since $E \subset \lambda B $, $\diam{\pi(E)} \leq 13 \lambda d$, which is a contradiction. Therefore  \eqref{eq:2pt_50C}  must hold.
\end{proof}

\section{Quasisymmetric  spheres over quasidisks}\label{sec:heightdistance}
First, we give the proofs of Theorem \ref{thm:LCA-QS-main} and Corollary \ref{cor:flatness-QSspheres-1/2}.

\begin{proof}[Proof of Theorem \ref{thm:LCA-QS-main}]
To prove the sufficiency, we apply Theorem \ref{thm:BonkKleiner} of  Bonk and Kleiner.  It has been shown, in Proposition \ref{prop:LLCiffLQC}, that if $\Omega$ has the LQC property then $\S(\Omega,\varphi)$ is LLC for all 
$\varphi \in \mathcal{F}$. Proposition \ref{prop:2-reg} in Section \ref{sec:2-regsection} below asserts  that if $\Omega$ has the LCA property then $\S(\Omega,\varphi)$ is Ahlfors $2$-regular for every $\varphi\in\mathcal{F}$. 
The sufficiency follows by Theorem \ref{thm:BonkKleiner}.

It has been shown in Proposition \ref{prop:LLCiffLQC} that if $\S(\Omega,\varphi)$ is LLC for all $\varphi \in \mathcal{F}$ then $\Omega$ must have the LQC property, in particular $\partial \Omega$ is a quasicircle. Proposition 
\ref{prop:chord-arcnec} in Section \ref{sec:vaismethod} states that if $\partial\Omega$ is a quasicircle but not a chord-arc curve then there exists a gauge function $\varphi\in\mathcal{F}$, necessarily depending on 
$\partial\Omega$, so that the associated double-dome-like surface $\S(\Omega,\varphi)$ fails to be quasisymmetric to $\mathbb{S}^2$. From these the necessity of the theorem follows.
\end{proof}

\begin{proof}[Proof of Corollary \ref{cor:flatness-QSspheres-1/2}]  Since $\partial \Omega$ is a chord-arc curve with flatness $\zeta_{\partial \Omega} < 1/2$, $\Omega$ has the LCA property by Theorem \ref{thm:VW-1/2}. Therefore, 
by Theorem \ref{thm:LCA-QS-main}, $\S(\Omega,\varphi)$ is a quasisymmetric sphere for every $\varphi \in \mathcal{F}$. On the other hand, there exists a domain $\Omega$ whose boundary is a chord-arc curve with flatness 
$\zeta_{\partial\Omega} = 1/2$ and which does not satisfy the LQC property; see \cite[Remark 5.2]{VW} for an example. By Lemma \ref{lem:lqc1}, $\S(\Omega,t^{\a})$ is not quasisymmetric to $\mathbb{S}^2$ for any $\a \in (0,1)$.
\end{proof}

The remaining part of this section is devoted to the proofs of Proposition \ref{prop:2-reg} and Proposition \ref{prop:chord-arcnec}.

\medskip

\subsection{Square pieces on $\S(\Omega,\varphi)$ with the assumption of LQC on $\Omega$}\label{sec:squarepieces}
Assume that a Jordan domain $\Omega$ has the  LQC property: there exist $C_0>1$ and $\e_0>0$ such that for any $ 0\leq \e\leq \e_0$
\begin{equation}\label{eq:LQCforsquare}
\diam{\g_\e(x,y)} \leq C_0|x-y|,  \quad \text{   for all }  x,y \in \g_\e.
\end{equation}

Unless otherwise mentioned, constants and comparison ratios in $\simeq$ and $\lesssim$ in this subsection depend at most on $C_0, \e_0$ and $\diam \Omega$.

\subsubsection{Quadrilaterals in $\Omega$}

A quadruple $\langle x_1,y_1,x_2,y_2 \rangle $ of (distinct) points  in $\overline{\Omega}$ is said to be \emph{admissible}, if
\begin{enumerate}
\item[(i)] $x_1, y_1\in \g_{t_1} \text{ and } x_2, y_2 \in \g_{t_2} \text{ for some }0\leq t_2<t_1\leq \e_0,$
\item[(ii)] $ |x_1- x_2| = \dist(x_1,\g_{t_2})=  |y_1-y_2| =\dist(y_1,\g_{t_2})=  t_1-t_2,$
\item[(iii)] $t_1-t_2 \leq |x_1-y_1| /3 \leq \diam{\Omega}/(10C_0)$.
\end{enumerate}
Note from the non-crossing and the monotonicity properties about the distance functions stated in Section \ref{sec:levelsets}, that the segments $[x_1, x_2]$ and $[y_1, y_2]$ do not meet and  that arcs  $\g_{t_1}(x_1, y_1)$ and 
$\g_{t_2}(x_2, y_2)$ intersect $[x_1,x_2] \cup [y_1,y_2]$ only at their end points. Denote by $Q(x_1,y_1,x_2,y_2)$ the quadrilateral whose boundary is the Jordan curve 
$[x_1,x_2]\cup [y_1,y_2] \cup \g_{t_1}(x_1,y_1) \cup \g_{t_2}(x_2,y_2)$.

Similarly, every $\g_t$, $t_2\leq t \leq t_1$ intersects  the segment $[x_1,x_2]$ (resp. $[y_1,y_2]$)  at precisely one point which we call $x^t$ (resp.$y^t$), and the arc $\g_t(x^t, y^t)$ is contained entirely in 
$Q(x_1,y_1,x_2,y_2)$. Note from (iii) and the LQC property that
\begin{equation}\label{eq:disk}
Q(x_1,y_1,x_2,y_2) \text{ contains a disk } B^2(z,(t_1-t_2)/4)
\end{equation}
centered on $ \g_{(t_1+t_2)/2}$, and that
\begin{equation}\label{eq:diameter}
\diam{\g_t(x^t, y^t)} \simeq  |x_1-y_1|
\end{equation}
and
\[ \diam{Q(x_1,y_1,x_2,y_2)} \simeq |x_1-y_1|. \]
for all  $t_2\leq t \leq t_1$.

\subsubsection{Square pieces}\label{sec:square-center}
Let $\varphi$ be a function in $\mathcal F$. A quadruple $\langle x_1,y_1,x_2,y_2 \rangle$  of points in $\S(\Omega,\varphi)^+$ is said to be \emph{admissible} if its projection quadruple 
$\langle \pi(x_1),\pi(y_1),\pi(x_2),\pi(y_2)\rangle$ is admissible in $\Omega$. In this case denote by  $D = D(x_1,y_1,x_2,y_2)$  the lift of $Q(\pi(x_1),\pi(y_1),\pi(x_2),\pi(y_2))$ on the surface $\S(\Omega,\varphi)^+$, i.e.,
\[ D(x_1,y_1,x_2,y_2)= \{(w,\varphi(\dist(w,\partial \Omega)))\colon w\in Q(\pi(x_1),\pi(y_1),\pi(x_2),\pi(y_2))\}. \]
By (i) and (ii),
\[|x_1-y_1|=|\pi(x_1)-\pi(y_1)|,\, |x_2-y_2|=|\pi(x_2)-\pi(y_2)|. \]

If $\langle x_1,y_1,x_2,y_2 \rangle$ is admissible in $\S(\Omega,\varphi)^+$  and satisfies, in addition,
\begin{enumerate}
\item[(iv)]\quad $\frac{1}{20 C_0}|x_1-y_1| \leq t_1-t_2+\varphi(t_1)-\varphi(t_2)  \leq \frac{1}{3}|x_1-y_1|,$
\end{enumerate}
we say  $\langle x_1,y_1,x_2,y_2 \rangle$ is \emph{admissible for a square piece} and call $ D(x_1,y_1,x_2,y_2)$ a \emph{square piece on $\S(\Omega,\varphi)^+$}. (Square pieces on $\S(\Omega,\varphi)^-$ can be defined 
analogously.) Note from the monotonicity of $\varphi$ that
\[ |y_1-y_2|= |x_1-x_2| \leq t_1-t_2 + \varphi(t_1) - \varphi(t_2) \leq \sqrt{2} |x_1-x_2|. \]
Also (iv) and the LQC property of $\Omega$ yield that the diameter of a square piece $D(x_1,y_1,x_2,y_2)$ is comparable to $|x_1-y_1|$:

\begin{equation}\label{eq:square-diam}
C_1^{-1}( t_1-t_2 + \varphi(t_1) - \varphi(t_2))\leq \diam{D(x_1,y_1,x_2,y_2)}  \leq C_1( t_1-t_2 + \varphi(t_1) - \varphi(t_2))
\end{equation}
for some constant $C_1>1$.

Fix a point $z= z(x_1,y_1,x_2,y_2)$, called a \emph{rough center} of the square piece $ D(x_1,y_1,x_2,y_2)$,  as follows. First, let $t_3$ be the unique number in $[t_2,t_1]$ satisfying
\[t_3-t_2 + \varphi(t_3) - \varphi(t_2)= \frac{1}{2}(t_1-t_2 + \varphi(t_1) - \varphi(t_2)),\]
and $\sigma$  be the lift of
\[ \g_{t_3} \cap Q(\pi(x_1),\pi(y_1),\pi(x_2),\pi(y_2)).\]
on $D(x_1,y_1,x_2,y_2) \subset \S(\Omega,\varphi) $. Then, take $z$ to be a point (choices may be plentiful)  on the arc $\sigma$ of equal distance to both endpoints of $\sigma$. It is straightforward to to check that
\begin{equation}\label{eq:center}
|z-w| \simeq \diam{D(x_1,y_1,x_2,y_2)}, \,\,\text{   for every } w\in \partial D(x_1,y_1,x_2,y_2).
\end{equation}

\begin{rem}\label{rem:contain}
Let $a$ be a point in $\S(\Omega,\varphi)^+$ with $\dist(\pi(a), \partial \Omega)<\e_0$ and let $0< r \leq \min\{\e_0/3 ,\frac{1}{200 C_0^2}\diam{\Omega}\}$. Then the surface $\S^+(\Omega,\varphi)\cap B^3(a,r)$ is contained in 
a square piece $D_1$ and contains a square piece $D_2$ with diameters comparable to $r$.
\end{rem}

To find $D_1$, we set
\[t_1=\max\{\dist(x,\partial \Omega) \colon x\in \pi(\S(\Omega,\varphi)^+ \cap \overline{B}^3(a,r))\},\]
\[ t_2=\min \{\dist(x,\partial \Omega) \colon x\in \pi(\S(\Omega,\varphi)^+ \cap \overline{B}^3(a,r))\}.\]
By the  monotonicity of $\varphi$,
\begin{equation}\label{eq:square-r}
r \leq t_1-t_2 + \varphi(t_1) - \varphi(t_2) \leq 2 \sqrt{2} r.
\end{equation}
Fix a point $z_1 \in \g_{t_1}$ whose lift $z_1^+$ on $\S(\Omega,\varphi)^+$ is contained in $\overline{B}^3(a,r)$. Choose  $x_1,y_1 \in \g_{t_1}$ such that $|x_1-z_1|= |y_1-z_1| = 8 C_0 r$. Let $x_2$ (resp. $y_2$) be a point on 
$\g_{t_2}$ that is closest to $x_1$ (resp. $y_1$). By the LQC property the quadruple $\langle x_1,y_1,x_2,y_2 \rangle$ is admissible in $\Omega$.  We check that the quadrilateral $Q(x_1,y_1,x_2,y_2)$ contains 
$\pi(\S(\Omega,\varphi)^+ \cap B^3(a,r))$ by showing  $|w-\pi(a)| > r$ for every  $w\in\g_t\setminus \g_t(x^t,y^t)$ with $t\in [t_2,t_1]$. Given $w\in\g_t\setminus \g_t(x^t,y^t)$ with $t\in [t_2,t_1]$, we fix a point  $z^t$ in 
$\g_{t}(x^t,y^t)$ that is closest to $z_1$. The arc $\g_t(w,z^t)$ contains one of the  points $x^t,y^t$; assume that it contains $x^t$. Note that
\[\diam{\g_t(w,z^t)} \geq |x^t-z^t| \geq |x_1-z_1| - |x_1-x^t| - |z^t- z_1| > 4 C_0r.\]
On the other hand, by the LQC property, $\diam{\g_t(w,w^t)}\leq C_0 |w-z^t|$. Hence $|w-\pi(a)| \geq |w-z^t|-|z^t-\pi (a)| >r$.

So $Q(x_1,y_1,x_2,y_2)$ lifts to a square piece  $D_1= D(x_1^+,y_1^+,x_2^+,y_2^+)$ which contains $ \S(\Omega,\varphi)^+ \cap B^3(a,r)$ and, by \eqref{eq:square-diam} and \eqref{eq:square-r}, has diameter
\[ \diam D_1 \leq 2 \sqrt{2}C_1 r;\]
recall that  $x_1^+,y_1^+,x_2^+,y_2^+$  are the lifts of $x_1,y_1,x_2,y_2$ on  $\S(\Omega,\varphi)^+$.

Define a square piece $D_2$ associated to $\S(\Omega,\varphi)^+ \cap B^3(a,r/(100 C_1))$ following the steps in constructing $D_1$ in the previous paragraph; here $C_1$ is the constant in \eqref{eq:square-diam}. So
\[\diam D_2 \leq \frac{20 \sqrt{2}C_1}{100 C_1} r < r/3.\]
Since $D_2$ intersects $ B^3(a,\frac{1}{100C_1}r)$, it is contained in $ \S(\Omega,\varphi)^+ \cap B^3(a,r)$.

\subsection{Ahlfors 2-regularity}\label{sec:2-regsection}
We prove in this section the following proposition which, combined with Theorem \ref{thm:BonkKleiner} and Lemma \ref{lem:LLC1}, concludes the sufficiency in Theorem \ref{thm:LCA-QS-main}.

\begin{prop}\label{prop:2-reg}
Let $\Omega$ be a planar Jordan that satisfies the level chord-arc property. Then $\S(\Omega,\varphi)$ is Ahlfors $2$-regular for every $\varphi \in \mathcal{F}$.
\end{prop}

The statement is quantitative. Suppose that $\Omega$ has the LCA property: there exists $c_0>1$ such that for every  $ 0\leq \e \leq \e_0$,
\begin{equation}\label{eq:LCAforsquare}
\ell({\g_\e(x,y)}) \leq c_0|x-y|,  \quad \text{ for all }  x,y \in \g_\e.
\end{equation}
Suppose also that $\varphi \in \mathcal F$ is  $L$-Lipschitz in $[\e_0/3,\infty)$. Then $\S(\Omega,\varphi)$ satisfies \eqref{eq:2regdefn} for some constant $C$ depending only on $c_0, \e_0,L$ and $\diam{\Omega}$.

\begin{proof}
Recall that $\D_{\e_0/3}$ consists of all points in $\Omega$ whose distance from $\partial \Omega$ is greater than  $\e_0/3$ and  $\D_{\e_0/3}^+$ is its lift on $\S(\Omega,\varphi)^+$. Since $\partial \D_{\e_0/3}=\g_{\e_0/3}$ 
is a $c_0$-chord-arc curve, the domain $\D_{\e_0/3}$ is the bi-Lipschitz image of the unit disk. The surface $\D_{\e_0/3}^+$, which is the graph of a $L$-Lipschitz function on  $\D_{\e_0/3}$,  is Ahlfors $2$-regular. Since the 
surface $\S(\Omega, \varphi)$ is symmetric with respect to $\R^2\times \{0\}$, it suffices to verify
\begin{equation}\label{eq:Ahlforsreg}
C^{-1} r^2  \le \mathcal{H}^2(B^3(a,r)\cap\S(\Omega,\varphi)^+) \leq C r^2,
\end{equation}
only for those points $a\in \S(\Omega,\varphi)^+$ whose projection has $\dist(\pi(a), \partial \Omega)<\e_0$ and for  $0< r \leq \min\{\e_0/3 ,\frac{1}{180 c_0^2}\diam{\Omega}\}$. This statement follows immediately from Remark 
\ref{rem:contain} and the area estimates \eqref{eq:LCA-area-D} for square pieces below.
\end{proof}

\subsubsection{Area of square pieces}
Let $\varphi\in \mathcal F$ and  $ D(x_1,y_1,x_2,y_2)$ be a square piece on $\S(\Omega,\varphi)^+$. We retain all assumptions and notations associated to the definitions of quadrilaterals and square pieces from the previous 
subsections.

First we observe that, assuming LQC on $\Omega$, the area of a square piece satisfies
\begin{equation}\label{eq:area-D}
\mathcal{H}^2(D(x_1,y_1,x_2,y_2)) \gtrsim |x_1-y_1|^2.
\end{equation}
When $t_1-t_2 \geq \varphi(t_1)-\varphi(t_2)$, we estimate $\mathcal{H}^2(D(x_1,y_1,x_2,y_2))$ from below by projecting $D(x_1,y_1,x_2,y_2)$ onto the plane $\R^2$ then applying  \eqref{eq:disk} and (iv). When 
$t_1-t_2 < \varphi(t_1)-\varphi(t_2)$, the lower estimate follows  from (iv), \eqref{eq:diameter} and the Fubini Theorem.

Next we claim that, assuming LCA on $\Omega$, the area of a square piece satisfies
\begin{equation}\label{eq:LCA-area-D}
C^{-1}|x_1-y_1|^2 \leq \mathcal{H}^2(D(x_1,y_1,x_2,y_2)) \leq C|x_1-y_1|^2,
\end{equation}
for some constant $C>1$ depending on $c_0,\e_0$ and $\diam \Omega$.

Note  by (iv) and \eqref{eq:LCAforsquare} that
\begin{equation}\label{eq:length}
\ell (\g_t(x^t, y^t))\simeq |x_1-y_1|,
\end{equation}
for all $t\in[t_2,t_1]$, recalling $0\leq t_2< t_1\leq \e_0$. To establish the upper bound, we let $\e \in (0,\frac{1}{3}|x_1-x_2|)$. Since $\varphi$ is monotone, the length of the  graph 
$\sigma=\{(t,\varphi(t))\colon t \in [t_2,t_1]\}$ of $\varphi$ over $[t_2,t_1]$ is at most $\sqrt{2} |x_1-x_2|$. Hence, there is a partition of $[t_2,t_1]$
\[ t_2 = \tau_n < \dots < \tau_{i+1} < \tau_i < \dots < \tau_0 = t_1 \]
such that length of the graph $\sigma_i$ of $\varphi$ over $[\tau_{i+1}, \tau_i]$ satisfies
\[ \e/4 \leq \ell(\sigma_i) \leq \e/2, \]
and therefore the number $n$ in the partition has an upper bound
\[ n \leq \frac{\ell(\sigma)}{\e/4} \leq \frac{\sqrt{2}|x_1-x_2|}{\e/4} \lesssim \frac{|x_1-y_1|}{\e}.\]
We next partition $D(x_1,y_1,x_2,y_2)$ into strips by the lifts $\g_{\tau_i}^+$ of the level curves $\g_{\tau_i}$. It follows from (iv) and \eqref{eq:length} that each strip can be covered by at most $C_2|x_1-y_1|/\e$ balls of 
radius $\e$. Therefore, $D(x_1,y_1,x_2,y_2)$ can be covered by at most $C_3|x_1-y_1|^2/\e^2$ balls of radius $\e$. This verifies the upper estimate and the claim \eqref{eq:LCA-area-D}.

\subsection{Chord-arc curves and V\"ais\"al\"a's method}\label{sec:vaismethod}
In this section we show that the chord-arc property of $\G$ is necessary for $\S(\Omega,\varphi)$ to be quasisymmetric sphere for all $\varphi \in\mathcal{F}$; this claim together with Proposition \ref{prop:LLCiffLQC} and 
Proposition \ref{prop:lqc+chord-arc} concludes the necessity part in Theorem \ref{thm:LCA-QS-main}.

\begin{prop}\label{prop:chord-arcnec}
Suppose that $\Omega$ is a quasidisk but $\G = \partial\Omega$ is not a chord-arc curve. Then there exists $\varphi \in\mathcal{F}$ such that $\S(\Omega,\varphi)$ is not quasisymmetric to $\mathbb{S}^2$.
\end{prop}

The main idea used in the proof is adapted from V\"ais\"al\"a \cite{Vais3}.

\begin{proof}
In view of  Proposition \ref{prop:LLCiffLQC}, we may assume that $\Omega$ has the LQC property with data $\e_0$ and $C_0$ as given in \eqref{eq:LQCforsquare}. Constants and comparison ratios in $\simeq$ and $\lesssim$ below 
depend at most on  $\e_0, C_0, \diam \Omega$ unless otherwise mentioned.

Since $\G$ is not a chord-arc curve, we may find a sequence $\{\G_n\}_{n\in \N}$ of subarcs of $\G$ such that $\ell(\G_n)/\diam{\G_n}\geq 2 n$ for all $n$ and $\diam{\G_n} \to 0$ as $n\to \infty$. For each $n\geq 1$, fix points 
$x_{n,0},x_{n,1}, \dots, x_{n,N_n}$ on $\G_n$,  labeled consecutively following their orientation in $\G_n$ and with $x_{n,0},  x_{n,N_n}$ representing the endpoints of $\G_n$, so that
\begin{equation}\label{eq:noCA}
\sum_{i=1}^{N_n} |x_{n,i-1}- x_{n,i}| \geq n\diam{\G_n}
\end{equation}
and that
\[ 1/2 \leq |x_{n,i-1} -x_{n,i}|/ |x_{n,j-1} -x_{n,j}| \leq 2,  \,\,\text{     for all } 1 \leq i, j \leq N_n.\]
After passing to a subsequence, we may assume that
\[d_n= \frac{1}{10 C_0^2}\min_{1\leq i \leq N_n }|x_{n,i-1}-x_{n,i}|\]
decrease to zero, as $n \to \infty$.

Define a homeomorphism $\varphi\colon [0,\infty) \to [0,\infty)$ in $\mathcal F$ with values
\[ \varphi(0)=0 \,\, \text{ and } \, \varphi (d_n) = \diam{\G_n} \,\,\, \text{for}\, n\ge 1,\]
and linear in $[d_1, \infty)$ and in each of the intervals $[d_n, d_{n-1}]$, $n\ge 1$.

We claim that $\S(\Omega,\varphi)$ is not quasisymmetric to $\mathbb{S}^2$. Assume the contrary: there exists an $\eta$-quasisymmetric homeomorphism from $\S(\Omega,\varphi)$ onto $\mathbb{S}^2$. After post-composing this map 
with a M\"obius transformation, we may find an $\eta$-quasisymmetric embedding $F$ from $\S(\Omega,\varphi)^+$ to $\mathbb{B}^2$.

Now fix $n$. Write $N=N_n$, $d=d_n$ and $x_i = x_{n,i}$ for simplicity.

Because $\G$ is a quasicircle, by Lemma \ref{lem:proximity} and its proof, we can find for each $i$ a point  $w_i \in \g_d$ that satisfies $d \leq |w_i-x_i| \leq 3C_0 d$. Choose next points $\dot{w}_i$ in $\G$ so that 
$|w_i-\dot{w}_i|=\dist(w_i, \G)=d$. Hence  $w_i, \dot{w}_i$ are in $\G \cap B^2(x_i,4 C_0 d)$. Because $\G$ is a quasicircle satisfying the $2$-point condition with the constant $C_0$ and the distance between any two 
consecutive points in  $\{x_0,x_1,\ldots, x_N\}$ is at least $10 C_0^2 d$, sets $\{B^2(x_i,4C_0 d)\cap\G\colon 0\leq i \leq N \}$ are mutually disjoint and points $\dot{w}_0,\dot{w}_1,\dots, \dot{w}_N$ follow the order of 
$x_0,x_1,\dots,x_N$ on $\G$ consecutively. Therefore, the segments $[w_i,\dot{w}_i]$ are pairwise non-crossing and the points $w_0,w_1,\dots,w_N$ are also in consecutive order on $\g_d$.

Observe again by the LQC property that for every $i =1,2, \dots,N$
\begin{equation}\label{eq:sizeofsquare-small}
\diam{\G(\dot{w}_{i-1},\dot{w}_i)} \simeq \diam{\g_{d}(w_{i-1},w_i,)}  \simeq |w_{i-1}-w_i| \simeq   |\dot{w}_{i-1}-\dot{w}_i| \simeq |x_{i-1}-x_i|\simeq d,
\end{equation}
and
\begin{equation}\label{eq:sizeofsquare-width}
\diam{\G(\dot{w}_0,\dot{w}_{N})} \simeq \diam{\g_{d}(w_0,w_N)} \simeq |w_0-w_N| \simeq |\dot{w}_0-\dot{w}_{N}|  \simeq  |x_0-x_N|  \simeq \varphi(d).
\end{equation}
Then
\begin{equation}\label{eq:sizeofsquare-height}
|{w}_i^+ - \dot{w}_i| \simeq \varphi(d) \,\,\,\text{    for  } i=0, 1,\ldots, N,
\end{equation}
recalling that ${w}_i^+$ is the lift of $w_i$ on $\S(\Omega, \varphi)^+$.

Note that, by \eqref{eq:sizeofsquare-small} and \eqref{eq:sizeofsquare-width}, $\langle w_0^+, w_N^+, \dot{w}_0, \dot{w}_{N} \rangle$ is admissible for a square piece $D=D(w_0^+, {w}_N^+, \dot{w}_0, \dot{w}_N)$ in 
$\S(\Omega, \varphi)^+$ satisfying  $\diam{D} \simeq \varphi(d)$. Here and in the remaining part of the proof, square pieces are understood to satisfy (i), (ii), (iii) and (iv), with the possibility that the constants in (iv) 
are altered but nevertheless  depend  at most on $\e_0, C_0, \diam \Omega$.

Note also that $\langle w_{i-1}, w_i, \dot{w}_{i-1}, \dot{w}_i \rangle$ are  admissible in $\Omega$. Let $Q_i $ be the quadrilateral bounded by 
$\G(\dot{w}_{i-1},\dot{w}_i) \cup [w_{i-1},\dot{w}_{i-1}] \cup [w_i,\dot{w}_i] \cup \g_{d}(w_{i-1},w_i)$ and $D_i = D(w_{i-1}^+, w_i^+, \dot{w}_{i-1}^+, \dot{w}_i^+ )$ its lift on $\G(\Omega,\varphi)^+$ (not necessarily square 
pieces). Then $\{D_i\colon 1\leq i \leq N\}$ have pairwise disjoint interiors.

The square piece $D$ is partitioned into $N$ essentially disjoint tall and narrow strips $D_i$, each of  which has height in the magnitude $\varphi(d)$ and width in the magnitude $d$. The ratio  $N d / \varphi(d)$ of the total 
width  of $D$ to the height is large in view of  \eqref{eq:noCA}. In other words $D$ resembles the product $\sigma \times I$  of a long arc $\sigma$ with a unit interval $I$. To complete the proof we follow  {\Vaisala}'s method 
in \cite[Theorem 4.2]{Vais3}.

Slice $D_i$ by parallel planes $H_j= \R^2\times \{j d\}$  with $j=0,1,\dots, m,$ and $H_{m+1}= \R^2\times \{\varphi(d)\}$,  where $m$ is the largest integer less than $\varphi(d)/d-1$. The resulting sets 
$D_{i,j}, 1\leq i \leq N $ and $1 \leq j \leq m+1,$ are square-like pieces associated to the admissible quadruples $\langle v_{i-1,j},v_{i,j},v_{i-1,j-1},v_{i,j-1}, \rangle$ and have diameter comparable to $d$. Here, $v_{p,q}$ 
denotes the intersection of the arc $[w_p,\dot{w}_p]^+$ and the plane $H_q$.

Let $z_{i,j}$ be a rough center of $D_{i,j}$ defined as in Section \ref{sec:square-center}. Then by \eqref{eq:center}
\[ |z_{i,j}-x|\simeq \diam D_{i,j} \simeq d \,\, \text{   for any } x\in \partial D_{i,j}.\]
Let $u \in \partial D_{ij}$ be a point at which
\[ |F(z_{i,j}) - F(u)| = \dist(F(z_{i,j}), \partial F(D_{i, j})) = r. \]
Since $|v_{i,j-1}-v_{i,j}| \simeq |v_{i,j}-z_{i,j}| \simeq |z_{i,j}-u|$, it follows from the $\eta$-quasisymmetry of $F$ that $\beta_{i,j}= F(v_{i,j-1})-F(v_{i,j})\lesssim r$. Hence 
$\beta_{i,j}^2 \lesssim \mathcal{H}^2(F(D_{i,j}))$. Here and in the rest of proof, constants may also depend on $\eta$.

Set
\[\beta = \min \{ |F(x)-F(y)|\colon x \in {\g_{d}(w_0,w_{N})}^+,\, y \in {\G(\dot{w}_0,\dot{w}_{N})} \}. \]
By Schwarz inequality,
\[\beta^2 \leq \left ( \sum_{j=1}^{m+1}\beta_{i,j}  \right )^2 \lesssim (m + 1)\mathcal{H}^2(F(D_{i})) \]
for $i=1,2,\ldots,N$. Note  from \eqref{eq:noCA} that $ n \varphi(d)\lesssim Nd$ and, by the choice of $m$, that $\varphi(d)/d -2 \leq m <\varphi(d)/d -1$. Summing over $i$ we get
\begin{equation}\label{eq:beta}
n \beta^2  \lesssim \mathcal{H}^2(F(D)).
\end{equation}

Let $p \in {\g_{d}(w_0,w_{N})}^+$  and $q \in {\G(\dot{w}_0,\dot{w}_{N})}$ be the points at which $\beta=|F(p)-F(q)|$ is realized. Then $|x-p| \lesssim \diam D \simeq |p-q|$  for any $x\in D$, which implies by quasisymmetry 
that $|F(x) - F(p)| \lesssim \beta$. Since $F(D)$ is planar,  $\mathcal{H}^2(F(D)) \lesssim \beta^2$. Hence by \eqref{eq:beta},  $n \beta^2  \lesssim  \beta^2$.

This is impossible since $\beta$ and the comparison ratio are independent of $n$. The proof is completed.
\end{proof}

\bibliographystyle{abbrv}
\bibliography{thesisref}

\end{document}